\documentclass[12pt]{amsart}
\usepackage{amsmath,amsthm,amsfonts,amssymb,latexsym, mathrsfs}
\usepackage{graphicx, xcolor} % Required for inserting images
\usepackage{hyperref,backref}
\title[Characters and Sylow subgroup abelianization]{Characters and Sylow $3$-subgroup abelianization}
\author[E. Giannelli]{Eugenio Giannelli}
\address[E. Giannelli]{Dipartimento di Matematica e Informatica ‘Ulisse Dini’, 50134 Firenze, Italy}
\email{eugenio.giannelli@unifi.it}

\author[N. Rizo]{Noelia Rizo}
\address[N. Rizo]{Departament de Matem\`atiques, Universitat de Val\`encia,
 46100 Burjassot, Val\`encia, Spain}
 \email{Noelia.Rizo@uv.es}
 
\author[A. A. Schaeffer Fry]{A. A. Schaeffer Fry} 
\address[A. A. Schaeffer Fry]{Dept. Mathematics - University of Denver, Denver, CO 80210, USA}
\email{mandi.schaefferfry@du.edu}

\author[C. Vallejo]{Carolina Vallejo}
\address[C. Vallejo]{Dipartimento di Matematica e Informatica ‘Ulisse Dini’, 50134 Firenze, Italy}
\email{carolina.vallejorodriguez@unifi.it}

\date{\today}
\thanks{The first and fourth-named authors are funded by the Italian national project PRIN 2022- 2022PSTWLB - Group Theory and Applications %- CUP B53D23009410006
 and, as members of the GNSAGA, both are grateful for the support of the {\em Istituto Nazionale di
Alta Matematica}. The second-named author is supported by Ministerio de Ciencia e
Innovaci\'on 
(Grant PID2022-137612NBI00 funded
by MCIN/AEI/10.13039/501100011033 and “ERDF A way of making Europe”) and a
CDEIGENT grant CIDEIG/2022/29 funded by Generalitat Valenciana.  The third-named author  gratefully acknowledges support from the National Science Foundation, Award No. DMS-2100912,  and her former institution, Metropolitan State University of Denver, which holds the award and allows her to serve as PI. The fourth-named author also acknowledges support from the Rita Levi-Montalcini 2019 programme. This material is based upon work supported by the National Science Foundation under Grant No.
DMS-1928930, while the authors were in residence at the Mathematical Sciences Research Institute in
Berkeley, California, during the summer of 2023.}

\begin{document}

\newtheorem{thm}{Theorem}[section]

\newtheorem{lem}[thm]{Lemma}
\newtheorem*{thmA}{Theorem A}
\newtheorem*{thmB}{Theorem B}

\newtheorem{prop}[thm]{Proposition}
\newtheorem{conj}[thm]{Conjecture}
\newtheorem{cor}[thm]{Corollary}
\newtheorem{que}[thm]{Question}
\newtheorem{rem}[thm]{Remark}
\newtheorem{defi}[thm]{Definition}

\newcommand{\6}{^}
\newcommand{\Maxn}{\operatorname{Max_{\textbf{N}}}}
\newcommand{\Syl}{\operatorname{Syl}}
\newcommand{\Lin}{\operatorname{Lin}}
\newcommand{\U}{\mathbf{U}}
\newcommand{\R}{\mathbf{R}}
\newcommand{\dl}{\operatorname{dl}}
\newcommand{\Con}{\operatorname{Con}}
\newcommand{\cl}{\operatorname{cl}}
\newcommand{\Stab}{\operatorname{Stab}}
\newcommand{\Aut}{\operatorname{Aut}}
\newcommand{\Ker}{\operatorname{Ker}}
\newcommand{\InnDiag}{\operatorname{InnDiag}}
\newcommand{\fl}{\operatorname{fl}}
\newcommand{\Irr}{\operatorname{Irr}}
\newcommand{\FF}{\mathbb{F}}
\newcommand{\EE}{\mathbb{E}}
\newcommand{\normal}{\trianglelefteq}
\newcommand{\sn}{\normal\normal}
\newcommand{\Bl}{\mathrm{Bl}}
\newcommand{\NN}{\mathbb{N}}
\newcommand{\N}{\mathbf{N}}
\newcommand{\bfC}{\mathbf{C}}
\newcommand{\bfO}{\mathbf{O}}
\newcommand{\bfF}{\mathbf{F}}
\def\GGG{{\mathcal G}}
\def\HHH{{\mathcal H}}
\def\HH{{\mathcal H}}
\def\irra#1#2{{\rm Irr}_{#1}(#2)}

\newcommand{\bG}{{\mathbf{G}}}
\newcommand{\bS}{{\mathbf{S}}}
\newcommand{\bH}{{\mathbf{H}}}
\newcommand{\bT}{{\mathbf{T}}}
\newcommand{\bL}{{\mathbf{L}}}
\newcommand{\type}{\operatorname}
\newcommand{\POmega}{{\mathrm {P\Omega}}}

\newcommand{\GO}{{\mathrm {GO}}}
\newcommand{\SO}{{\mathrm {SO}}}
\newcommand{\PCO}{{\mathrm {PCO}}}

\newcommand{\wt}{\widetilde}

\renewcommand{\labelenumi}{\upshape (\roman{enumi})}

\newcommand{\PSL}{\operatorname{PSL}}
\newcommand{\PSU}{\operatorname{PSU}}
\newcommand{\alt}{\operatorname{Alt}}

\providecommand{\V}{\mathrm{V}}
\providecommand{\E}{\mathrm{E}}
\providecommand{\ir}{\mathrm{Irm_{rv}}}
\providecommand{\Irrr}{\mathrm{Irm_{rv}}}
\providecommand{\re}{\mathrm{Re}}

\numberwithin{equation}{section}
\def\irrp#1{{\rm Irr}_{p'}(#1)}

\def\ibrrp#1{{\rm IBr}_{\Bbb R, p'}(#1)}
\def\C{{\mathbb C}}
\def\Q{{\mathbb Q}}
\def\irr#1{{\rm Irr}(#1)}
\def\irrp#1{{\rm Irr}_{p^\prime}(#1)}
\def\irrq#1{{\rm Irr}_{q^\prime}(#1)}
\def \c#1{{\cal #1}}
\def \aut#1{{\rm Aut}(#1)}
\def\cent#1#2{{\bf C}_{#1}(#2)}
\def\norm#1#2{{\bf N}_{#1}(#2)}
\def\zent#1{{\bf Z}(#1)}
\def\syl#1#2{{\rm Syl}_#1(#2)}
\def\normal{\triangleleft\,}
\def\oh#1#2{{\bf O}_{#1}(#2)}
\def\Oh#1#2{{\bf O}^{#1}(#2)}
\def\det#1{{\rm det}(#1)}
\def\gal#1{{\rm Gal}(#1)}
\def\ker#1{{\rm ker}(#1)}
\def\normalm#1#2{{\bf N}_{#1}(#2)}
\def\alt#1{{\rm Alt}(#1)}
\def\iitem#1{\goodbreak\par\noindent{\bf #1}}
   \def \mod#1{\, {\rm mod} \, #1 \, }
\def\sbs{\subseteq}

\def\gc{{\bf GC}}
\def\ngc{{non-{\bf GC}}}
\def\ngcs{{non-{\bf GC}$^*$}}
\newcommand{\notd}{{\!\not{|}}}

\newcommand{\Z}{\mathbf{Z}}
\newcommand{\Out}{{\mathrm {Out}}}
\newcommand{\Mult}{{\mathrm {Mult}}}
\newcommand{\Inn}{{\mathrm {Inn}}}
\newcommand{\IBR}{{\mathrm {IBr}}}
\newcommand{\IBRL}{{\mathrm {IBr}}_{\ell}}
\newcommand{\IBRP}{{\mathrm {IBr}}_{p}}
\newcommand{\cd}{\mathrm{cd}}
\newcommand{\ord}{{\mathrm {ord}}}
\def\id{\mathop{\mathrm{ id}}\nolimits}
\renewcommand{\Im}{{\mathrm {Im}}}
\newcommand{\Ind}{{\mathrm {Ind}}}
\newcommand{\diag}{{\mathrm {diag}}}
\newcommand{\soc}{{\mathrm {soc}}}
\newcommand{\End}{{\mathrm {End}}}
\newcommand{\sol}{{\mathrm {sol}}}
\newcommand{\Hom}{{\mathrm {Hom}}}
\newcommand{\Mor}{{\mathrm {Mor}}}
\newcommand{\Mat}{{\mathrm {Mat}}}
\def\rank{\mathop{\mathrm{ rank}}\nolimits}
\newcommand{\Tr}{{\mathrm {Tr}}}
\newcommand{\tr}{{\mathrm {tr}}}
\newcommand{\Gal}{{\rm Gal}}
\newcommand{\Spec}{{\mathrm {Spec}}}
\newcommand{\ad}{{\mathrm {ad}}}
\newcommand{\Sym}{{\mathrm {Sym}}}
\newcommand{\Char}{{\mathrm {Char}}}
\newcommand{\pr}{{\mathrm {pr}}}
\newcommand{\rad}{{\mathrm {rad}}}
\newcommand{\abel}{{\mathrm {abel}}}
\newcommand{\PGL}{{\mathrm {PGL}}}
\newcommand{\PCSp}{{\mathrm {PCSp}}}
\newcommand{\PGU}{{\mathrm {PGU}}}
\newcommand{\codim}{{\mathrm {codim}}}
\newcommand{\ind}{{\mathrm {ind}}}
\newcommand{\Res}{{\mathrm {Res}}}
\newcommand{\Lie}{{\mathrm {Lie}}}
\newcommand{\Ext}{{\mathrm {Ext}}}
\newcommand{\Alt}{{\mathrm {Alt}}}
\newcommand{\AAA}{{\sf A}}
\newcommand{\SSS}{{\sf S}}
\newcommand{\DDD}{{\sf D}}
\newcommand{\QQQ}{{\sf Q}}
\newcommand{\CCC}{{\sf C}}
\newcommand{\SL}{{\mathrm {SL}}}
\newcommand{\Sp}{{\mathrm {Sp}}}
\newcommand{\PSp}{{\mathrm {PSp}}}
\newcommand{\SU}{{\mathrm {SU}}}
\newcommand{\GL}{{\mathrm {GL}}}
\newcommand{\GU}{{\mathrm {GU}}}
\newcommand{\Spin}{{\mathrm {Spin}}}
\newcommand{\CC}{{\mathbb C}}
\newcommand{\CB}{{\mathbf C}}
\newcommand{\RR}{{\mathbb R}}
\newcommand{\QQ}{{\mathbb Q}}
\newcommand{\ZZ}{{\mathbb Z}}
\newcommand{\bfN}{{\mathbf N}}
\newcommand{\bfZ}{{\mathbf Z}}
\newcommand{\PP}{{\mathbb P}}
\newcommand{\cG}{{\mathcal G}}
\newcommand{\OO}{\mathcal O}
\newcommand{\cH}{{\mathcal H}}
\newcommand{\cQ}{{\mathcal Q}}
\newcommand{\GA}{{\mathfrak G}}
\newcommand{\cT}{{\mathcal T}}
\newcommand{\cL}{{\mathcal L}}
\newcommand{\IBr}{\mathrm{IBr}}
\newcommand{\cS}{{\mathcal S}}
\newcommand{\cR}{{\mathcal R}}
\newcommand{\GCD}{\GC^{*}}
\newcommand{\TCD}{\TC^{*}}
\newcommand{\FD}{F^{*}}
\newcommand{\GD}{G^{*}}
\newcommand{\HD}{H^{*}}
\newcommand{\GCF}{\GC^{F}}
\newcommand{\TCF}{\TC^{F}}
\newcommand{\PCF}{\PC^{F}}
\newcommand{\GCDF}{(\GC^{*})^{F^{*}}}
\newcommand{\RGTT}{R^{\GC}_{\TC}(\theta)}
\newcommand{\RGTA}{R^{\GC}_{\TC}(1)}
\newcommand{\Om}{\Omega}
\newcommand{\eps}{\epsilon}
\newcommand{\varep}{\varepsilon}
\newcommand{\al}{\alpha}
\newcommand{\chis}{\chi_{s}}
\newcommand{\sigmad}{\sigma^{*}}
\newcommand{\PA}{\boldsymbol{\alpha}}
\newcommand{\gam}{\gamma}
\newcommand{\lam}{\lambda}
\newcommand{\la}{\langle}
\newcommand{\genf}{F^*}
\newcommand{\ra}{\rangle}
\newcommand{\hs}{\hat{s}}
\newcommand{\htt}{\hat{t}}
\newcommand{\tG}{\hat G}
\newcommand{\St}{\mathsf {St}}
\newcommand{\bfs}{\boldsymbol{s}}
\newcommand{\bfl}{\boldsymbol{\lambda}}
\newcommand{\tn}{\hspace{0.5mm}^{t}\hspace*{-0.2mm}}
\newcommand{\ta}{\hspace{0.5mm}^{2}\hspace*{-0.2mm}}
\newcommand{\tb}{\hspace{0.5mm}^{3}\hspace*{-0.2mm}}
\def\skipa{\vspace{-1.5mm} & \vspace{-1.5mm} & \vspace{-1.5mm}\\}
\newcommand{\tw}[1]{{}^#1\!}
\newcommand{\Irrg}[1]{\Irr_{p',\sigma}(#1)}
\renewcommand{\mod}{\bmod \,}

\maketitle
\begin{abstract}
We characterize when a finite group $G$ possesses a Sylow $3$-subgroup $P$ with abelianization  of order $9$ in terms of the number of height zero characters lying in the principal 
$3$-block of $G$, settling a conjecture put forward by Navarro, Sambale, and Tiep in 2018.  Along the way, we show that a recent result by Laradji on the number of character of height zero in a block that lie above a given character of some normal subgroup holds, without any hypothesis on the group for blocks of maximal defect.
\end{abstract}

\section{Introduction}

\noindent A key role in the theory of finite groups is played  by dihedral, semidihedral and generalized quaternion $2$-groups. These can be characterized, for example, as the non-abelian $2$-groups $P$ whose commutator subgroup $P'$ has index $4$ in $P$, by work of O. Taussky \cite{Tau37}. Since the beginning of the $20$th Century, the representation theory of finite groups having a Sylow $2$-subgroup $P$ with $[P:P']=4$ (and more generally of blocks having a defect group $P$ with $[P:P']=4$) has been largely studied. Partly relying on this extensive study, G. Navarro, B. Sambale, and P. H. Tiep showed in \cite[Thm.~A]{NST18} that a finite group $G$ possesses a Sylow $2$-subgroup $P$ whose commutator subgroup has index $4$ if, and only if, the principal $2$-block of $G$ contains exactly $4$ irreducible characters of odd degree. 

Contrasting this situation, the representation theory of groups having a Sylow $p$-subgroup $P$ with $[P:P']=p^2$ for an odd prime $p$ is far from being understood.
 It is perhaps not so surprising, given the fact that the structure of these $p$-groups for a general prime can be exceedingly complicated.  For instance, write $f_p(n)$ to denote the number of isomorphism classes of groups $P$ of order $p^n$ with $[P:P']=p^2$. For $p=2$, as mentioned above, we have that $f_2(n)=3$ if $n\geq 4$, but already for $p=3$ the function $f_3(n)$ tends to infinity with $n$ (see for instance \cite{Neb89}). 
 In fact, proving that a finite group has a Sylow $3$-subgroup $P$ with $[P:P']=9$ if, and only if, its 
principal $3$-block contains exactly $6$ or $9$ irreducible characters of degree coprime to $3$ has remained a challenge since this statement was put forward in \cite[Conj.~B]{NST18}. The main result of this article is to prove the above mentioned conjecture.

\begin{thmA} Let $G$ be a finite group and $P\in \syl 3 G$. 
Then $[P:P']=9$ if, and only if, there are exactly $6$ or $9$ irreducible characters of degree coprime to $3$ lying in the principal $3$-block of $G$.
\end{thmA}

For a group $G$, we write $k_0(B_0(G))$ to denote the number of irreducible characters of degree coprime to $3$ lying in the principal $3$-block $B_0(G)$ of $G$ (as these are exactly those of $3$-height zero in $B_0(G)$). Theorem A above, together with \cite[Thm.~A]{NST18}, imply that the conditions $[P:P']\leq p^2$ and $k_0(B_0(G))\leq p^2$ are equivalent for any finite group $G$ and $p\in \{ 2, 3\}$. This statement does not hold for arbitrary primes as the group $H=\verb|SmallGroup(1000,163)|$ satisfies $k(H)=k_0(B_0(H))=25$ with respect to the prime $p=5$ and its Sylow $5$-subgroups are abelian of order $125$. We care to mention that, for a general prime $p$, it has been proven very recently in \cite{NS23} that the local condition $[P:P']=p^2$ can be detected in the character table of the group $G$. However, it is still not known whether there exists a simple
algorithm to do that for primes $p\geq 5$.

We also mention that Theorem A is an essential step toward proving the $p=3$ analogue of the main result of \cite{NRSV21}, which would give a way to determine from the character table whether a Sylow $3$-subgroup is $2$-generated.

\smallskip

Our proof of Theorem A eventually reduces to showing that the statement of Theorem A holds for almost simple groups $A$ with $[A: {\rm Soc}(A)]_3\leq 3$, where $n_3$ is the largest power of $3$ dividing an integer $n$ and ${\rm Soc}(A)$ is the socle of $A$.  On the way to proving Theorem A for these almost simple groups, we classify when $[P:P']=9$ for those groups and prove some additional statements that could be of independent interest.  
In order to perform the reduction, we have also proven results of independent interest regarding the behavior of characters in principal blocks with respect to normal subgroups and the character theory of simple non-abelian groups with respect to the action of group automorphisms on them. The former appear in Section \ref{Sec:principalblocks} and the latter in Section \ref{Sec:almostsimple} of this work. 
Among the results briefly mentioned above, we would like to highlight the following generalization of the main theorem of \cite{Lar23}.

\begin{thmB} Let  $p$ be a prime. Let $G$ be a finite group, $P\in \syl p G$ and $B$ a $p$-block of $G$ with defect group $P$.
Let $N$ be a normal subgroup of $G$ such that $PN/N>1$. If $\theta \in {\rm Irr}(N)$, then  the number $k_0(B|\theta)$ of irreducible characters of degree coprime to $p$ in $B$ lying above $\theta$ is either 0 or $k_0(B|\theta)\geq 2$. If we further assume that $p\leq 3$,
then $k_0(B|\theta)$ is divisible by $p$.
\end{thmB}

The structure of the paper is as follows. In Section 2, we state some results (both old and new) on principal blocks  (and more generally, blocks of maximal defect). In Section 3, we state our results on almost simple groups, which will be proven in Sections 4 and 5. We address the alternating groups in Section 4 and the groups of Lie type in Section 5. Finally, we complete the proof of Theorem A in Section 6.

\section{General statements on principal blocks}\label{Sec:principalblocks}

\subsection{Known results on principal blocks}
For a finite group $G$ and a prime $p$, we denote by $B_0(G)$ the principal $p$-block of $G$. 
Recall that 
$${1}_G \in \irr{B_0(G)}=\{ \chi \in \irr G \ | \ \sum_{{x \in G}\atop{(p, o(x))=1}}\chi(x)\neq 0 \}\, .$$
In particular, if $G=H_1\times\dots\times H_t$, then 
$$\Irr(B_0(G))=\{\theta_1\times\dots\times\theta_t\mid\theta_i\in\Irr(B_0(H_i))\}.$$
Given a normal subgroup $N$ of $G$, we can view the irreducible characters of $G/N$ as irreducible characters of $G$ containing $N$ in their kernel, and under this identification
$$\irr {B_0(G/N)}\sbs \irr{B_0(G)},$$
with equality in the case where the order of $N$ is coprime to $p$ by \cite[Thm.~9.9(c)]{Navarro98}.

For any $\theta\in\irr{B_0(N)}$, there is some $\chi\in\irr{B_0(G})$ lying over $\theta$  (that is, with $[\chi, \theta^G]\neq0$), and for any  $\chi\in\irr{B_0(G)}$, every constituent of $\chi_N$ lies in $B_0(N)$ by \cite[Thm.~9.2, Cor.~9.3, and Thm.~9.4]{Navarro98}. If moreover $B_0(G)$ is the only block covering $B_0(N)$, then for any  $\theta\in\irr{B_0(N)}$ every constituent of $\theta^G$ lies in $B_0(G)$. This happens, for instance, in the following case.
\medskip

\begin{lem}\label{lemma:onlyblockabove} Let $G$ be a finite group and let $p$ be a prime. Suppose that $M$ is a normal subgroup of $G$ containing $\cent G Q$, where $Q \in \syl p M$. Then the principal $p$-block $B_0(G)$ of $G$ is the only block covering the principal $p$-block $B_0(M)$ of $M$.
\end{lem}
\begin{proof}
Let $B$ be a $p$-block of $G$ covering $B_0(M)$.  By \cite[Thm.~9.26]{Navarro98}, $Q$ is contained in some defect group $P$ of $B$. Since $\cent G P \leq \cent G Q \leq M$, by \cite[Lem.~9.20]{Navarro98} $B$ is regular with respect to $M$. 
By \cite[Thm.~9.19]{Navarro98}, $B_0(M)^G$ is defined and hence
$B=B_0(M)^G=B_0(G)$
by Brauer's Third Main Theorem \cite[Thm.~6.7]{Navarro98}.
\end{proof}

Next is a well-known result of   J. Alperin  and E. C. Dade. We will use both the notation ${\rm Irr}_{p'}(B_0(G))$ and ${\rm Irr}_0(B_0(G))$ to denote the set of irreducible characters of height zero ($p'$-degree) in the principal $p$-block of $G$. Moreover we let $k(G)=|\mathrm{Irr}(G)|$, $k(B_0(G))=|{\rm Irr}(B_0(G))|$ and $k_0(B_0(G))=|{\rm Irr}_0(B_0(G))|$.

\begin{thm}[Alperin-Dade]\label{alperin-dade}
Suppose that $N$ is a normal subgroup of $G$, with $G/N$ a $p'$-group.
Let $P \in \syl p G$ and assume that $G=N\cent GP$. Then restriction of characters defines
a natural bijection between the irreducible characters of the principal blocks
of $G$ and $N$. In particular, $k_0(B_0(G))=k_0(B_0(N))$.
\end{thm}

\begin{proof}
The case where $G/N$ is solvable was proved in \cite{Alperin76} and the general case
in \cite{Dad77}. The last part of the statement follows from the first.
\end{proof}

The following is contained in \cite[Lem.~4.3]{Murai94}.

\begin{lem}[Murai]\label{lem:Murai} Let $p$ be a prime. Let $N$ be a normal subgroup of $G$ and let $\theta$ be a character of height zero in a $p$-block $b$ of maximal defect of $N$. Let $P\in{\rm Syl}_{p}(G)$. Suppose that $\theta$ extends to $PN$. If $B$ is a $p$-block of $G$ of maximal defect covering $b$, then there exists $\chi\in{\rm Irr}(B)$ of height zero lying over $\theta$.
\end{lem}

For a normal subgroup $N$ of $G$ and given $\theta \in {\rm Irr}(B_0(N))$, we will sometimes write ${\rm Irr}_0(B_0(G)|\theta)=\{ \chi \in {\rm Irr}_0(B_0(G)) \ | \ [\chi_N, \theta]\neq 0\}$ and $k_0(B_0(G)|\theta)$ to denote the size of the above set. Lemma \ref{lem:Murai} above guarantees that $k_0(B_0(G)|\theta)\neq 0$ if $\theta$ has height zero and extends to $PN$, where $P\in \syl p G$.

\subsection{New results on principal blocks} 
Here we collect some new properties of principal blocks that are key to our reduction process and that we believe to be of independent interest.

\smallskip

If $\chi \in {\rm Irr}(B_0(G))$ and a normal subgroup $N$ is contained in the kernel of $\chi$, then in general it is not true that $\chi$, seen as a character of $G/N$, lies in $B_0(G/N)$.
The following result was actually proven within the proof of \cite[Thm.~3.5]{NRSV21}. We include it here for the reader's convenience. We recall that $\Phi(P)$ denotes the Frattini subgroup of a $p$-group $P$.

\begin{thm}\label{thm:NRSV}
Let $N$ be a normal subgroup of a finite group $G$. Suppose that $N\sbs \Phi(P)$, where $P\in \syl p G$. Write $B_0=B_0(G)$ to denote the principal $p$-block of $G$. If $\chi$ is a height zero irreducible character in $B_0$ and $N\subseteq {\rm Ker}(\chi)$, then $\chi$ seen as a character of $G/N$ belongs to the principal $p$-block of $G/N$.
\end{thm}

\begin{proof}

 Write $\bar\chi$ to denote the character $\chi$ seen as a character of $G/N$. Since $\bar\chi$
 has degree not divisible by $p$, it follows that $\bar\chi$ lies in a block of $G/N$ of maximal defect $P/N$.
 By  \cite[Prob.~4.5]{Navarro98}, we only need to prove that if $Nx \in G/N$, with $x \in G$, is $p$-regular
 with $P/N\subseteq \cent{G/N}{Nx}$, then
 $$\left( {[G/N : \cent{G/N}{Nx}]  \bar\chi(Nx) \over  \bar \chi(1)}\right) \equiv [G/N : \cent{G/N}{Nx}]$$ modulo any maximal ideal $\mathcal P$ of the ring $\bf R$ of algebraic integers in $\mathbb C$ containing $p$.
 First notice that $Nx \in \norm{G/N}{P/N}= \norm GP/N$. Also, we may assume that $x$ is
 a $p$-regular element, using that $Nx=(Nx)_{p'}=Nx_{p'}$.  Notice that
 $$\cent{G/N}{Nx}=\cent Gx N/N \, ,$$ using that $(|N|, o(x))=1$ (and \cite[Cor.~3.28]{Isa08}). 
We have shown that $x$ is a $p$-regular element of
 $\norm GP$ centralizing $P/N$. Therefore $x$ centralizes $P/\Phi(P)$,
 and by 
 \cite[Cor.~3.29]{Isa08} 
 we have that $x$ centralizes $P$. 
 In particular, $N \subseteq \cent Gx$. Then $[G/N : \cent{G/N}{Nx}]=[G:\cent Gx]$.
 Since $\chi$ is in the principal block of $G$, we have that 
 $$\left( {[G/N : \cent{G/N}{Nx}]  \bar\chi(Nx) \over  \bar \chi(1)}\right)=
 \left( {[G : \cent{G}{x}]  \chi(x) \over   \chi(1)}\right)\equiv [G: \cent Gx]\, ,$$
 modulo any maximal ideal $\mathcal P$ of the ring of algebraic integers $\bf R$ in $\mathbb C$ containing $p$, as desired.
\end{proof}

If $N \triangleleft G$ and $\theta$ is a character in the principal block of $N$ that extends to $G$, in general it is not true that the principal block of $G$ must contain some extension of $\theta$. 
The following guarantees the existence of an extension in the principal block. The main idea of the proof is taken from the proof of \cite[Thm.~6.26]{Isaacs76}. We denote by ${\rm Lin}(G)$ the subset of ${\rm Irr}(G)$ of characters of $G$ of degree 1, often called linear characters.

\begin{lem}\label{lem:extensionprincipalblock} Let $N$ be a normal $p$-subgroup of $G$. Suppose that $\lambda \in {\rm Lin}(N)$ is $G$-invariant and extends to $P$, where $P\in \syl p G$.
Then $\lambda$ extends to some character in the principal block of $G$.
\end{lem}
\begin{proof} Write $|G|=p^a m$ with $(m, p)=1$. Let $\delta \in {\rm Lin}(G)$ be an extension of $\lambda$ (such an extension exists by \cite[Theorem 6.26]{Isaacs76}). 
Choose a positive integer $b$ such that $m$ divides $b$ and $b\equiv 1 \mod o(\lambda)$. 
Notice then that  $\delta^b \in {\rm Lin}(G)$.  Write $G^0=\{ g \in G \ | \ (o(g), p)=1\}$ to denote the set of $p$-regular elements of $G$. 
Since every $g \in G^0$ satisfies $o(g)$ divides $m$, we have that $\delta(g)$ is an $m$th root of unity. In particular $\delta^b(g)=1$ for every $g \in G^0$. 
Since then $$\displaystyle \sum_{x \in G^0}\delta^b(x)=|G^0|\neq 0\, ,$$
we have that $\delta^b$  lies in the principal $p$-block of $G$ by \cite[Thm.~3.19]{Navarro98}.
\end{proof}

We need to prove a relative version of classic results by Landrock \cite{Landrock81} and Cliff--Plesken--Weiss \cite{CPW87} (stated below for the reader's convenience) for principal blocks. In fact, we can do more than that.  Below, $k_0(B)$ denotes the number of height zero characters in a block $B$.

\begin{thm}\label{Landrock} Let $B$ a $p$-block of nontrivial defect of a finite group $G$. Then 
$$k_0(B)\geq 2\, .$$ 
If moreover $p\leq 3$, then 
$p$ divides $k_0(B)$.
\end{thm}
\begin{proof}
The first part appears in \cite{CPW87} (see also \cite[Prob.~3.11]{Navarro98}), and the second is \cite[Cors.~1.3 and 1.6]{Landrock81} (see also \cite[Lem.~2.2]{NST18} for a proof for principal blocks).
\end{proof}

Below $k_0(B|\theta)=|{\rm Irr}_0(B)\cap {\rm Irr}(G|\theta)|$ for a block $B$ and a character $\theta$ of some normal subgroup of $G$. We next prove Theorem B from the Introduction. We re-state it here in a way in which $k_0(B|\theta)\neq 0$ (by using Lemma \ref{lem:Murai}). 
  
\begin{thm}\label{thm:thetatheorem}
    Let $G$ be a finite group, $p$ be a prime and $P\in \syl p G$. Let $N$ be a normal subgroup of $G$. Suppose that  $PN/N>1$. Let $B$ be a $p$-block of $G$ with defect group $P$ covering a $p$-block $b$ of $N$ with defect group $N\cap P$.
    Let $\theta\in{\rm Irr}_0(b)$. Suppose that 
    $\theta$ extends to $PN$. Then $k_0(B|\theta)\geq 2$. If moreover $p\leq 3$, then 
     $p$ divides $k_0(B|\theta)$.
\end{thm}
\begin{proof}
First notice that, by induction on $[G:N]$, and using the Clifford correspondence together with the Fong--Reynolds correspondence \cite[Thm.~9.14]{Navarro98}, we may assume that $\theta$ is $G$-invariant.

Define $$\psi=\sum_{\chi \in {\rm Irr}(G|\theta)}\chi(1)\chi\, .$$
Notice that $\psi(x)=0$ if $x \notin N$ and $\psi(x)=|G/N|\theta(1)\theta(x)$ if $x \in N$ by a theorem of Kn\"orr (see \cite[Thm.~5.21]{Navarro18}). In particular, $\psi^*\equiv 0$, where $^*\colon {\bf R} \to \mathbb F$ is the ring homomorphism defined in \cite[Ch.~2]{Navarro98}, naturally extended to characters of $G$ (recall that $\bf R$ is the ring of algebraic integers in $\C$ and we choose a maximal ideal $\mathcal P$ containing $p$ so that $\mathbb F= {\bf R}/\mathcal P$ is algebraically closed of characteristic $p$ and $^*$ is the natural epimorphism). 

We now follow the proof of \cite[Thm.~5.10]{Navarro98} to show that $(\psi_{C})^*$ vanishes on the set $G^0$ of $p$-regular elements of $G$ for every $p$-block $C$ of $G$, where $\displaystyle \psi_{C}=\sum_{\chi \in {\rm Irr}(C|\theta)}\chi(1)\chi$.
Given $y \in G^0$ (so $y \in G$ such that $(o(y), p)=1$), we have that
$$\psi_{C}(y)=\sum_{\xi \in {\rm Irr}(C)}[\psi, \xi]\xi(y)=\sum_{\xi \in {\rm Irr}(C)}[\psi, \xi]\left(\sum_{\varphi \in {\rm IBr(G)}}d_{\xi, \varphi}\varphi(y)\right)\, .$$
Here, following the notation of \cite[Chapter 2]{Navarro98} we used the symbol $d_{\xi, \varphi}$ to denote the corresponding decomposition number.
In particular, 
$$\psi_{C}(y)=\sum_{\varphi \in {\rm IBr}(G)}\left(\sum_{\xi \in {\rm Irr}(C)}[\psi, \xi]d_{\xi, \varphi}\right)\varphi(y)\, .$$
As $\psi=\sum_{C \in {\rm Bl}(G)}\psi_C$, we then have that
$$0=\psi(y)^*=\sum_{\varphi \in {\rm IBr}(G)}\left(\sum_{\xi \in {\rm Irr}(G)}[\psi, \xi]d_{\xi, \varphi}\right)^*\varphi(y)^*\, .$$
As the trace functions $\{ \varphi^* \ | \ \varphi \in {\rm IBr}(G)\}$ are $\mathbb F$-linearly independent, by \cite[Thm.~1.19]{Navarro98},
we conclude that for any $\varphi \in {\rm IBr}(G)$ we have that
$$\left(\sum_{\xi \in {\rm Irr}(G)}[\psi, \xi]d_{\xi, \varphi}\right)^*=0\, .$$
Now if $\varphi$ does not lie in $C$ then $$\sum_{\xi \in {\rm Irr}(C)}[\psi, \xi]d_{\xi, \varphi}=0\, .$$
If $\varphi$ lies in $C$ then  
$$0=\left(\sum_{\xi \in {\rm Irr}(G)}[\psi, \xi]d_{\xi, \varphi}\right)^*=\left(\sum_{\xi \in {\rm Irr}(C)}[\psi, \xi]d_{\xi, \varphi}\right)^*\, .$$
In particular, we have shown that $(\psi_{C}(1))^*=0$. In other words, the positive integer
$$\sum_{\chi \in {\rm Irr}(C |  \theta)}\chi(1)^2$$
is divisible by $p$, so is the positive integer
$$\sum_{\chi \in {\rm Irr}_0(C |  \theta)}\chi(1)^2\, ,$$
where we are using Lemma \ref{lem:Murai} to see that it is positive. This implies the first part of the statement. If $p\leq 3$, then $m^2 \equiv 1 \mod p$ for every integer $m$ coprime to $p$, so the second part of the statement also follows.
\end{proof}

As mentioned in the Introduction, notice that above we have proven that, for blocks of maximal defect, the main result of \cite{Lar23} holds without the $p$-solvability hypothesis on $G$.

\smallskip

We close this part with an easy to prove yet useful observation.

\begin{lem}\label{lem:extension} 
   Let $G$ be a finite group and $p$ be a prime. Let $N$ be a normal subgroup of $G$. Let $\theta\in{\rm Irr}_0(B_0(N))$. Suppose that $p$ divides $[G:N]$ and 
    $\theta$ extends to $G$. 
Then $k_0(B_0(G)|\theta)\geq k_0(\bar B)$ for some block $\bar B$ of maximal defect of $G/N$. 
If $\theta$ extends to some $\chi \in {\rm Irr}(B_0(G))$ then $k_0(B_0(G)|\theta)\geq k_0(B_0(G/N))$.
%If moreover $p\leq 3$, then $k_0(B_0(G)|\theta)$ is divisible by $p$.
\end{lem}
\begin{proof}
Write simply $B_0=B_0(G)$. Let $P/N\in{\rm Syl}_p(G/N)$. By hypothesis,  $P/N\in{\rm Syl}_p(G/N)$ and $\theta$ extends to $P$. By Theorem \ref{thm:thetatheorem} $k_0(B_0|\theta)\geq 2$ and if $p\leq 3$ then $p$ divides $k_0(B_0|\theta)$. 

Let $\tilde{\theta}\in{\rm Irr}(G)$ be an extension of $\theta$ and let $\chi \in {\rm Irr}_{p'}(B_0|\theta)$. By Gallagher's theory, $\chi=\beta\tilde{\theta}$ for some  $\beta\in{\rm Irr}_{p'}(G/N)$. Now the $p$-block $\bar{B}$ of $G/N$ containing $\beta$ is of maximal defect.
Using  \cite[Lem.~2.4]{Riz19}, for instance, we have that for every  $\xi\in{\rm Irr}_{p'}(\bar{B})$ the character $\xi\tilde{\theta}\in{\rm Irr}_0(B_0|\theta)$. Hence  $k_0(B_0|\theta)\geq k_0(\bar{B})$.  If we can choose $\tilde{\theta}$ in $B_0$ then $\bar B=B_0(G/N)$ by \cite[Lem.~5.1]{NV17}, for instance. 
\end{proof}

\section{Statements for almost simple groups}\label{Sec:almostsimple}

Our goal in the next several sections is to prove Theorems \ref{thm:conditionsonsimples} and \ref{thm:theoremAforalmostsimple} stated below. The latter is  Theorem A from the Introduction for almost simple groups $A$ with $[A : {\rm Soc}(A)]_3\leq 3$.

\begin{thm}\label{thm:conditionsonsimples} Let $S$ be a nonabelian simple group of order divisible by $3$.  Then:
\begin{enumerate}
    \item[(a)] Let $X\in{\rm Syl}_3({\rm Aut}(S))$. One of the following holds:
\begin{enumerate}
\item[(a1)] There exist $1_S\neq \theta_1,\theta_2\in{\rm Irr}_0(B_0(S))$ nonconjugate in ${\rm Aut}(S)$ and $X$-invariant; or
\item[(a2)] there exist an $X$-invariant $1_S\neq \theta\in{\rm Irr}_0(B_0(S))$  such that for  every $S\leq T \leq {\rm Aut}(S)_\theta$, we have $\theta$ extends to some character in $B_0(T)$.
\end{enumerate}

\item[(b)] The set of degrees of $3'$-irreducible characters in the principal $3$-block
 of $S$ has size at least 3.

\item[(c)] If $S$ has order divisible by $9$, then there exist $1_S\neq \theta_1,\theta_2, \theta_3\in{\rm Irr}_0(B_0(S))$ non-conjugate in ${\rm Aut}(S)$. 
\end{enumerate}
\end{thm}

\begin{rem}\label{remark:conditions}
We will see that the only cases in which we rely on (a2) rather than (a1) are when $S=\PSL_2(q)$ or when $3\nmid(q-\epsilon)$ and $S= \PSL_3^\epsilon(q)$. In those (a2) cases, $\Out(S)$ has a normal cyclic Sylow $3$-subgroup.
For $p\geq 5$, the $p$-version of condition (b) above was shown to hold in \cite{GRSS20}. Note that condition (c) above holds if $S$ has abelian Sylow $3$-subgroups by \cite[Thm.~2.1]{HSV24}, so when proving (c) we may assume that $S$ has non-abelian Sylow $3$-subgroups.
\end{rem}

\begin{thm}\label{thm:theoremAforalmostsimple} Let $A$ be an almost simple group with $[A:{\rm Soc}(A)]_3\leq 3$. Write $k_0(B_0)$ to denote the number of irreducible characters of degree coprime to $3$ in the principal $3$-block of $A$ and $P\in \syl 3 A$.
Then $[P:P']=9$ if, and only if, $k_0(B_0)\in \{ 6, 9 \}$.
\end{thm}

\begin{prop}\label{prop:spor}
Theorems \ref{thm:conditionsonsimples} and \ref{thm:theoremAforalmostsimple} hold when $S$ is a simple sporadic group,  a group of Lie type with exceptional Schur multiplier, or one of:  $\tw{2}\type{G}_2(3)'$, $\tw{2}\type{F}_4(2)'$, $\PSL_2(27)$, $\PSL_2(81)$, $\PSL_3^\epsilon(3)$, $\PSL_4^\epsilon(3)$, or $\PSp_4(3)$.
\end{prop}
\begin{proof}
A list of exceptional Schur multipliers may be found, for example, in \cite[Table~24.3]{MT11}. The result follows from computation in  \cite{GAP}, also taking into account Propositions \ref{prop:SylowSideNondef} and \ref{prop:nondefmostA3sylow} below for some groups of Lie type with exceptional Schur multipliers. 
\end{proof}

The following appeared within the proof of \cite[Thm.~2.4]{GRSS20}. It shows how we will use the statement of Theorem \ref{thm:conditionsonsimples}.(a2) in the context of minimal normal non-abelian subgroups. 

\begin{thm}\label{tensorinduction}
Suppose that $N\triangleleft\, G$ is a direct product $N=S_1\times \cdots \times S_t$ of simple non-abelian groups of order divisible by $p$ transitively permuted by $G$ by conjugation.
Let $\theta=\theta_1\times \cdots \times \theta_t \in {\rm Irr}(N)$ be $G$-invariant. If $\theta_1$ extends to some irreducible character of degree coprime to $p$ in the principal $p$-block of $T$ for every  $S_1\leq T \leq {\rm Aut}(S_1)_{\theta_1}$, then $\theta$ extends to some $\chi \in {\rm Irr}_{p'}(B_0(G))$. 
\end{thm}
\begin{proof} We may view $\theta_1$ as a character of $S_1 {\bf C}_G (S_1)/{\bf C}_G (S_1)$. Of course, we have that $S_1 {\bf C}_G (S_1)/{\bf C}_G (S_1)$ is normal in ${\bf N}_G (S_1)/{\bf C}_G (S_1)$. By hypothesis, $\theta_1$ extends to $\hat {\theta_1}\in {\rm Irr_{p'}}(B_0({\bf N}_G (S_1)/{\bf C}_G (S_1))) $.
Let $\displaystyle M=\bigcap_{i=1}^t {\bf N}_G(S_i) \subseteq {\bf N}_G(S_1)$. We have that $\hat{\theta_1}|_M\in {\rm Irr}_{p'}(B_0(M))$. Consider $\chi =\hat{\theta_1}^{\otimes G}$ the tensor induction of $\hat \theta_1$ as in \cite[Section 10.2]{Navarro18}. By  \cite[Cor.~10.5]{Navarro18}, we have that $\chi_N= \theta$. Notice that $\chi_M$ is a product of $G$-conjugates of $\hat{\theta_1}_M$ (see the proof of \cite[Thm.~2.4]{GRSS20} for further details), and so $\chi _M \in {\rm Irr}_{p'}(B_0(M))$.
Let $Q=P\cap M\in{\rm Syl}_p(M)$ and let $Q_i=Q\cap S_i\in {\rm Syl}_p(S_i)$. Since $\cent G Q\subseteq\cent G {Q_i}\subseteq\norm {G}{S_i}$, we have that $\cent G {Q}\subseteq M$ and then $B_0(G)$ is the only block covering $B_0(M)$ by Lemma \ref{lemma:onlyblockabove}. Therefore, $\chi\in{\rm Irr}_{p'}(B_0(G))$, as wanted.
\end{proof}

The following result will come in handy when proving the case $[A:{\rm Soc}(A)]_3=1$ of Theorem \ref{thm:theoremAforalmostsimple}.

\begin{prop} \label{prop:Noelialemma}
Let $p=3$, let $S$ be a nonabelian simple group of order divisible by $p$, and let $S\leq A\leq{\rm Aut}(S)$ with $[A:S]_3=1$. Suppose that there exists an $A$-invariant nontrivial $\theta\in{\rm Irr}_{p'}(B_0(S))$ and that there are at least 5 $A$-orbits of characters in ${\rm Irr}_{p'}(B_0(S))$. Suppose that $k_0(B_0(A))\in\{6,9\}$. Then $k_0(B_0(S))\leq 18$.
\end{prop}

\begin{proof}
In this case $P\in{\rm Syl}_p(S)$ and $A=S\norm A P$ by the Frattini argument, so $M=S\cent A P$ is normal in $A$. By Lemma \ref{lemma:onlyblockabove}, $B_0(A)$ is the only $p$-block of $A$ covering $B_0(M)$, so in particular $k(A/M)\leq k_0(B_0(A))\leq 9$, but since there are at least 5 orbits of characters in ${\rm Irr}_{p'}(B_0(S))$ and $[A:M]_3=1$, then $k(A/M)\in\{1,2,4,5\}$. If $A=M$, by Theorem \ref{alperin-dade} we have that $k_0(B_0(S))=k_0(B_0(A))$ and we are done. So we may assume $A>M$. Let $\tilde{\theta}\in{\rm Irr}_{p'}(B_0(M))$ be the extension of $\theta$ to $M$ given by Theorem \ref{alperin-dade}. We claim that $\tilde{\theta}$ is $A$-invariant. Indeed, let $a\in A$ and notice that $\tilde{\theta}^a\in{\rm Irr}_{p'}(B_0(M))$ is also an extension of $\theta^a=\theta$. Since restriction is a bijection in Theorem \ref{alperin-dade}, we have that $\tilde{\theta}^a=\tilde{\theta}$ and the claim is proved.

Suppose $k(A/M)=2$. In this case $A/M\cong C_2$. Since $1_M$ and $\tilde\theta$ extend to $A$, we have at most 5 other $A$-orbits in ${\rm Irr}_{p'}(B_0(M))$ of size at most 2, so $k_0(B_0(S))=k_0(B_0(M))\leq 12$.

Next, suppose that $k(A/M)=4$. Since $|A:M|_3=1$, then $A/M\in\{C_4,C_2\times C_2, D_{10}\}$. Since $\tilde{\theta}$ is $A$-invariant, if $A/M=D_{10}$, or $A/M=C_4$, we would have that $\tilde{\theta}$ extends to $A$, and we have too many characters in ${\rm Irr}_{p'}(B_0(A))$. So necessarily $A/M=C_2\times C_2$ and ${\rm Irr}(A|\tilde{\theta})=\{\chi\}$, then there are at most 4 other orbits of size at most 4, and then $k_0(B_0(S))\leq 18$, as wanted.

Finally, suppose that $k(A/M)=5$, so $A/M\in\{C_5,D_8,Q_8,D_{14},C_5\rtimes C_4\}$ (again we are using that $|A:M|$ is not divisible by 3 to exclude some groups having 5 conjugacy classes). In this case, every nontrivial $A$-orbit in ${\rm Irr}_{p'}(B_0(M))$ has a unique $\chi\in{\rm Irr}(A)$ over it, so in particular $\tilde{\theta}$ is fully ramified in $A$. But then $[A:M]$ needs to be a square, a contradiction.
\end{proof}

\section{Alternating and Symmetric Groups}\label{sec:alternating}

Here we first show that Theorem \ref{thm:conditionsonsimples} holds for simple alternating groups. Then,  we show that Theorem \ref{thm:theoremAforalmostsimple} holds for symmetric and alternating groups. 
Given a partition $\lambda$ of a natural number $n$, we let $\chi\6\lambda$ be the corresponding irreducible character of the symmetric group $S_n$. Moreover, we let $C_3(\lambda)$ and $\lambda'$ denote the $3$-core and the conjugate partition of $\lambda$, respectively. We refer the reader to \cite{Olsson93} for a complete account of the basic facts concerning the combinatorics of partitions.

Let $n\geq 5$ be a natural number and let $A_n$ denote the alternating group of rank $n$. We recall that $\mathrm{Aut}(A_n)=S_n$ for all $n\neq 6$. On the other hand, $\mathrm{Aut}(A_6)/A_6\cong C_2\times C_2$. In particular, if $X\in \mathrm{Syl}_3(\mathrm{Aut}(A_n))$ then $X\leq A_n$ and every irreducible character of $A_n$ is obviously $X$-invariant.

Let $a\in \{0,1,2\}$ be such that $n=3w+a$, for some $w\in\mathbb{N}$. Let $\lambda_1$ and $\lambda_2$ be the partitions of $n$ defined as follows: 
$$\lambda_1= \begin{cases}
		(n-1,1) & \text{if } a=0,\\
		 (n-2,2) & \text{if } a=1,\\
		  (2,2, 1\6{n-4}) & \text{if } a=2.
	\end{cases}\ \ \lambda_2=\begin{cases}
		(n-2,1\62) & \text{if } a=0,\\
		 (n-4,2, 1\62) & \text{if } a=1,\\
		  (2, 1\6{n-2}) & \text{if } a=2.
	\end{cases}$$
For any $i\in\{1,2\}$, it is easy to verify that $C_3(\lambda_i)=(a)$ and that $\lambda_i\neq \lambda_i'$. It follows that $\chi\6{\lambda_i}\in\mathrm{Irr}(B_0(S_n))$ and therefore that $\theta_i:=(\chi\6{\lambda_i})_{A_n}\in\mathrm{Irr}(B_0(A_n))$. 
An easy application of the hook length formula \cite[20.1]{James} shows that: 
$$\theta_1(1)= \begin{cases}
		n-1 & \text{if } a=0,\\
		 n(n-3)/2 & \text{if } a=1,\\
		  n(n-3)/2 & \text{if } a=2.
	\end{cases}\ \ \theta_2(1)=\begin{cases}
		(n-1)(n-2)/2 & \text{if } a=0,\\
		 (n\62-2n)(n-3)(n-5)/8 & \text{if } a=1,\\
		  n-1 & \text{if } a=2.
	\end{cases}$$

It follows that $\theta_1(1)$ and $\theta_2(1)$ are greater than $1$, coprime to $3$, and distinct integers.
This in turn implies that conditions ($a_1$) and (b) of Theorem \ref{thm:conditionsonsimples} hold for simple alternating groups. 

In order to prove that condition (c) holds as well, we can assume that $n\geq 6$ as $9$ should divide the order of $A_n$. We further assume that $n\geq 7$, as the case $n=6$ can be easily treated by hand (it is easy to see that $\mathrm{Irr}_0(B_0(A_6))$ contains $3$ characters of degrees $5$, $8$ and $10$ respectively). 
We recall that Olsson's theory of $3$-core towers \cite[Cor.~11.8]{Olsson93} allows us to give a precise formula for $k_0(B_0(S_n)))$. In particular, if $n=\sum_{k=0}\6ta_k3\6{k}$ is the $3$-adic expansion of $n$ then we have that $$k_0(B_0(S_n))=\prod_{k=1}\6tc(3\6k, a_k).$$
Here $c(M, a)$ denotes the number of sequences of length $M$ of the form $(\gamma_1, \ldots, \gamma_M)$ where $\gamma_1,\ldots, \gamma_M$ are $3$-core partitions whose sizes sum to $a$.
In particular, we deduce that $k_0(B_0(S_n)))\geq c(3\6t, a_t)\geq c(3\6t, 1)=3\6t$, this in turn implies that $k_0(B_0(S_n)))\geq 9$, whenever $n\geq 9$. Hence, for $n\geq 9$ there are at least $5$ non $\mathrm{Aut}(A_n)$-conjugate irreducible characters of degree coprime to $3$ in the principal block of $A_n$. For a slightly different reason the same argument applies to the cases $n\in\{7,8\}$. Here we have that $n=2\cdot 3 + a$ for some $a\in\{1,2\}$. It is easy to check that $c(3, 2)=9$. Exactly as before we deduce that there are at least $5$ non $\mathrm{Aut}(A_n)$-conjugate irreducible characters of degree coprime to $3$ in the principal block of $A_n$. We conclude that condition (c) of Theorem \ref{thm:conditionsonsimples} holds for $A_n$. 

We end this section with the verification of Theorem \ref{thm:theoremAforalmostsimple} for symmetric and alternating groups. 
For $k\in\mathbb{N}$ let $X_k$ denote a Sylow $3$-subgroup of $S_{3\6k}$. We know that $X_k\cong C_3\wr C_3\wr\cdots\wr C_3$ the iterated wreath product of $k$ copies of the cyclic group of order $3$. 
It is not difficult to show that $[X_k: X_k']=3\6k$ (it is very easy to show that $X_k$ has exactly $3\6k$ linear characters. A slightly more involved exercise is to show that $X_k/ X_k'$ is isomorphic to the direct product of $k$ copies of $C_3$). 
Let $n=\sum_{k=0}\6ta_k3\6{k}$ be the $3$-adic expansion of $n$ as above. Given $X\in\mathrm{Syl}_3(S_n)=\mathrm{Syl}_3(A_n)$, we know from \cite[4.1.22]{JK} that $$X=\prod_{k=1}\6t (X_k)\6{\times a_k}.$$
Hence, we have that $[X: X']=\prod_{k=1}\6{t}3\6{k\cdot a_k}$. If we now assume that $[X:X']=9$ we immediately deduce that $6\leq n\leq 11$. Using \cite[Cor.~11.8]{Olsson93} and \cite[Prop.~4.4]{OlssonZZZ} we obtain that 
$\{k_0(B_0(S_n))\ |\ 6\leq n\leq 11\}=\{9\}$ and that $\{k_0(B_0(A_n))\ |\ 6\leq n\leq 11\}=\{6, 9\}$. This proves one implication of Theorem \ref{thm:theoremAforalmostsimple}. 
The converse is similar. The formulas provided in \cite[Cor.~11.8]{Olsson93} and \cite[Prop.~4.4]{OlssonZZZ} imply that if $k_0(B_0(A_n)), k_0(B_0(S_n))\in\{6,9\}$ then $n\in\{6,7,\ldots, 11\}$. As discussed above we know that in these cases we have that $[X:X']=9$.

\section{Groups of Lie type}

Let $p$ be a prime and let $\bG$ be a connected reductive algebraic group over $\bar\FF_p$. Let $F\colon \bG\rightarrow\bG$ be a Steinberg endomorphism. Then the group of fixed points $G=\bG^F$ is a group of Lie type. In this situation, we will write $G^\ast$ for the group $(\bG^\ast)^F$ with $(\bG^\ast, F)$ in duality with $(\bG, F)$ as in \cite[Def.~1.5.17]{GM20}. 

We summarize some relevant Deligne--Lusztig theory, for which we refer the reader to, e.g. \cite[Sec.~2.6]{GM20}. The set $\irr{G}$ is partitioned into rational Lusztig series $\mathcal{E}(G,s)$ where $s$ ranges over the semisimple elements of $G^\ast$, up to $G^\ast$-conjugacy. When $s=1$, the characters in the set $\mathcal{E}(G,1)$ are known as unipotent characters. The set $\mathcal{E}(G,s)$ is in bijection with the set of unipotent characters of $\cent{G^\ast}{s}$, and we refer to those corresponding to unipotent characters of $\cent{G^\ast}{s}$ lying above the trivial character of $(\cent{\bG^\ast}{s}^\circ)^F $ as semisimple characters. In particular, we assume that a choice of this Jordan decomposition has been made and will denote by $\chi_s\in\mathcal{E}(G,s)$ the semsisimple character corresponding to $1_{\cent{G^\ast}{s}}$. (Although, when $\cent{\bG^\ast}{s}$ is connected, no choice is necessary.) The degree of $\chi\in\mathcal{E}(G, s)$ corresponding to $\psi\in\mathcal{E}(\cent{G^\ast}{s}, 1)$ is $[G^\ast: \cent{G^\ast}{s}]_{p'}\psi(1)$, so that $\chi_s(1)=[G^\ast: \cent{G^\ast}{s}]_{p'}$. 

By 
\cite[Prop.~8.26 and (8.19)]{CE04}, we have a bijection  $\zent{G^\ast}\rightarrow\irr{G/\textbf{O}^{p'}(G)}$, where if $z\mapsto \hat z$, then $\mathcal{E}(G,s)\otimes \hat z=\mathcal{E}(G, sz)$.

\subsection{Our Setting}\label{sec:setting}

We will work primarily in the following setting. Let $S$ be a simple group of Lie type defined in characteristic $p$ such that there is a simple, simply connected algebraic  group $\bG$ over $\bar\FF_p$ and a Steinberg endomorphism $F\colon \bG\rightarrow\bG$ such that $S=G/\zent{G}$ with $G=\bG^F$. (Recall that the Tits group $\tw{2}\type{F}_4(2)'$ was considered in Proposition \ref{prop:spor}, so we may omit this in our setting.)
We can write $\Aut(S)=\wt{S}\rtimes \mathcal{D}$, where $\wt{S}$ is the group of inner-diagonal automorphisms of $S$ and $\mathcal{D}$ is an appropriate group generated by graph and field automorphisms. (See, e.g. \cite[Thm.~24.24]{MT11}.) In particular, we will write $F_p$ for a generating field automorphism induced by the map $x\mapsto x^p$ on $\bar\FF_p$ and $F_{p^i}=F_p^i$ for a positive integer $i$.  Further, it will sometimes be useful to identify $\wt{S}$ with the fixed points $\bG_{\mathrm{ad}}^F$ of the adjoint group of the same type as $\bG$. (See \cite[Prop.~24.21]{MT11} and the subsequent discussion.) If $\varphi\in\Aut(S)$, then a series $\mathcal{E}(G, s)$ is mapped to a series $\mathcal{E}(G, s^{\varphi^\ast})$ for an appropriate dual automorphism $\varphi^\ast$ of $G^\ast$, by \cite[Prop.~7.2]{taylor}.

We will also often use that, from \cite[Lem.~4.4 and Rem.~4.6]{NT13}, we have the characters in $\mathcal{E}(G,s)$ for $s\in [G^\ast, G^\ast]$ are trivial on $\zent{G}$.

\smallskip

For the moment, let $\ell$ be any prime and write $B_0(X)$ for the principal $\ell$-block of a group $X$. 
Now, when $\ell\neq p$, that is $\ell$ is not the defining characteristic of $G$, a unipotent block is one containing unipotent characters, and the set of characters lying in unipotent blocks is exactly the union of the series $\mathcal{E}(G,s)$ with $s\in G^\ast$ an $\ell$-element, by \cite[Thm.~9.12]{CE04}. In particular, $\irr{B_0(G)}$ is contained in this union. By  \cite[Cor.~3.4]{hiss} and \cite[Thm.~21.13]{CE04}, the character $\chi_s$ is contained in $\irr{B_0(S)}$ whenever $s$ is a semisimple $\ell$-element and either $\zent{\bG^\ast}$ is connected or $\ell$ is good for $\bG$ and $\cent{\bG^\ast}{s}$ is connected. 

On the other hand, when $\ell=p$ is the defining characteristic, we have $\irr{B_0(S)}=\irr{S}\setminus\{\mathrm{St}_S\}$ by \cite[Thm.~6.18]{CE04}.

The following observations regarding extensions will  be useful.

\begin{lem}\label{lem:tildeScoprime}
Keep the setting above. Let $\ell\neq p$ be a non-defining prime dividing $|S|$ but not dividing $|\wt{S}/S|$.  Then restriction provides a bijection between $\irr{B_0(\wt{S})}$ and $\irr{B_0(S)}$.
\end{lem}
\begin{proof}
Let $\chi\in\irr{B_0(S)}$ and let $\wt\chi\in\irr{B_0(\wt{S})\mid\chi}$, which exists by \cite[Thm.~9.2]{Navarro98}. Since $\wt{S}/S$ is abelian and restrictions from $\wt{S}$ to $S$ are multiplicity-free (see e.g. \cite[Thm.~15.11]{CE04}), we have by Clifford theory that the number of constituents of $\wt\chi|_S$ is the number of $\hat{z}\in\irr{\wt{S}/S}$ such that $\wt\chi\hat{z}=\wt\chi$ (see e.g. \cite[Lem.~1.4]{RSV21}). %(see, e.g. \cite[Prop.~15.12]{CE04}). 
By \cite[Thm.~9.12]{CE04}, the characters in $B_0(\wt{S})$ lie in Lusztig series indexed by semisimple $\ell$-classes of $(\bG_{\mathrm{ad}}^\ast)^F$ and as discussed above, 
 the irreducible characters of $\wt{S}/S=\bG_{\mathrm{ad}}^F/\textbf{O}^{p'}(\bG_{\mathrm{ad}}^F)$ are in bijection with the elements $z$ of $\zent{\bG_{\mathrm{ad}}^\ast}^F$, and tensoring with $\hat z$ sends the series $\mathcal{E}(\wt{S}, s)$ to $\mathcal{E}(\wt{S}, sz)$. Since $\ell$ does  not divide the order of any such $z$, it follows that $\wt\chi|_S=\chi$. Now, since $\wt{S}/S$ is abelian, every character of $\wt S$ above $\chi$ is an extension, and hence of the form $\wt{\chi}\hat z$ for some $z\in\zent{\bG_{\mathrm{ad}}^\ast}^F$.  Again, since $\ell$ does not divide the order of  any such $z\neq 1$,  we therefore see that $B_0(\wt{S})$ contains only one extension of $\chi$.
\end{proof}

\begin{lem}\label{lem:Tlemma}
Keep the setting above and assume $S\neq \type{D}_4(q)$ and let $\ell$ be a prime dividing $|S|$. Suppose $\wt \chi\in\irr{B_0(\wt{S})}$ extends to $\Aut(S)_{\wt \chi}$ and restricts irreducibly to $\chi\in\irr{S}$. Then for every every $S\leq T \leq {\rm Aut}(S)_{\wt \chi}$, we have $\chi$ extends to some character in $B_0(T)$. 
\end{lem}
\begin{proof}
First, with our assumption on $S$, we have $\Aut(S)/\wt{S}$  and $\wt{S}/S$ are both abelian. Then every character in $\Irr(\Aut(S)_{\wt\chi}\mid \wt\chi)$ is an extension, and hence there is an extension  $\hat\chi$  of $\wt\chi$ in $B_0(\Aut(S)_{\wt\chi})$ by \cite[Thm.~9.4]{Navarro98}.
Now, let $S\leq T\leq \Aut(S)_{\wt{\chi}}$.  Then $\hat{\chi}|_{T\cap\wt{S}}=\wt\chi|_{T\cap\wt{S}}$ lies in $B_0(T\cap \wt{S})$ since $B_0(\wt{S})$ must cover a unique block of the normal subgroup $T\cap \wt{S}$.
By \cite[Thm.~9.4]{Navarro98}, $B_0(T)$ contains a character lying above $\hat{\chi}|_{T\cap\wt{S}}$.  Since $T/(T\cap\wt{S})$ is abelian, again this character is an
extension, completing the proof.
\end{proof}

For the remainder of the section, we will return to the case $\ell=3$, so that $B_0(X)$ denotes the principal $3$-block of a group $X$.

\subsection{Defining Characteristic}

Here we deal with the case that $p=3$. That is, $S$ is defined in characteristic $3$.

\begin{prop}\label{prop:defininginitial}
Let $q$ be a power of $3$ and let $S$ be one of the simple groups $\type{G}_2(q)$, $\type{D}_4(q)$, $\tw{3}\type{D}_4(q)$, $\type{F}_4(q)$, $\type{E}_6^\pm(q)$, $\type{E}_7(q)$, or $\type{E}_8(q)$. Let $X\in\Syl_3(\Aut(S))$. 
Then there exist at least four $X$-invariant characters with distinct degrees in $\irra{3'}{B_0(S)}$. Further, if $S=\tw{2}\type{G}_2(q^2)$, with $q^2=3^{2n+1}>3$, then there are six distinct character degrees in $\irra{3'}{B_0(S)}$ and at least three not $\Aut(S)$-conjugate, nontrivial $X$-invariant characters, two of which have distinct degrees.
In particular, Theorem \ref{thm:conditionsonsimples} (with (a1)), 
holds in these cases.
\end{prop}
\begin{proof}
Recalling that $B_0(S)=\Irr(S)\setminus\{\mathrm{St}_S\}$, we have $\Irr_{3'}(B_0(S))=\Irr_{3'}(S)$. Keeping the notation of Sec.~\ref{sec:setting}, we have  
$\wt{S}=S=G$ unless $S=\type{E}_7(q)$, in which case $|\wt{S}/S|=|\zent{G}|=2$ or $S=\type{D}_4(q)$, in which case $|\wt{S}/S|=|\zent{G}|=4$.
 The character degrees of $\wt{S}$ and $G$ and their multiplicities have been computed by L{\"u}beck and are available at \cite{lubeckwebsite}. From this information, we see that in each case other than $\tw{2}\type{G}_2(q^2)$, there are at least three nontrivial $3'$-character degrees of $\wt S$ whose multiplicities are not divisible by $3$ and whose constituents on $S$ have distinct degrees of multiplicity not divisible by $3$. When $S=\tw{2}\type{G}_2(q^2)$ with $q^2>3$, we see five distinct nontrivial character degrees, two of which have multiplicity prime to $3$. We see directly from the character table in \cite{ward} that there are at least three nontrivial, non-$\Aut(S)$-conjugate, $X$-invariant characters in this case.   
\end{proof}

\begin{lem}\label{lem:typeAdefining}
Let $S=\PSL_n^\epsilon( q)$, where $q$ is a power of $3$ and $n\geq 2$. Then Theorem \ref{thm:conditionsonsimples} holds for $S$. Further, (a1) holds for $n\geq 4$.
\end{lem}
\begin{proof}

Note that in this case, $XS/S$ must be generated by field automorphisms.
As before, $\Irr_{3'}(B_0(S))=\Irr_{3'}(S)$ and we may assume by Proposition \ref{prop:spor}  that $S$ has a nonexceptional Schur multiplier.

First suppose $S=\PSL_2(q)$ with $q\geq 9$.  In this case, the character table for $S$ is well-known, even available in GAP4 \cite{GAP}. The two characters $\chi_1, \chi_2$ of degree $\frac{q\pm1}{2}$ must be invariant under $X\in\Syl_3(\Aut(S))$ but fuse in $\wt{S}$. Here $\mathcal{D}$ is cyclic, so $\Aut(S)_{\chi_1}/S$ is cyclic. Then every character lying above $\chi_1$ in  $\Aut(S)_{\chi_1}$ is an extension, so some  extension $\hat\chi$ lies in $B_0(\Aut(S)_{\chi_1})$. For $S\leq T\leq \Aut(S)_{\chi_1}$, we then have the restriction of $\hat\chi$ to $T$ lies in $B_0(T)$ and extends $\chi_1$, giving (a2). There are also characters of degree $q+1$ and $q-1$  in $\Irr(S)$, giving (b) and (c).

We now assume $n\geq 3$. If $(n,q)\in\{(3,3), (4,3)\}$, we apply Proposition \ref{prop:spor}, so we may further assume that $q\geq 9$ or $n\geq 5$. From \cite[Thm.~6.8]{malle07}, we see that $\Irr_{3'}(B_0(S))$ consists of the deflations of the semisimple  characters $\chi_s$ of $G=\SL_n^\epsilon(q)$ that are trivial on $\zent{G}$. Recall that such a character has degree $[G^\ast: \cent{G^\ast}{s}]_{3'}$. 
Then for (b) and (c), it suffices by the discussion at the beginning of the section to exhibit three nontrivial classes of semisimple $s\in G^\ast$ satisfying  $s\in [G^\ast, G^\ast]$, and with distinct $|\cent{G^\ast}{s}|_{3'}$. We further write $\wt{G}:=\GL_n^\epsilon(q)$, so $[\wt{G}^\ast, \wt{G}^\ast]\cong G$. We will in fact exhibit three nontrivial $\wt{s}\in [\wt{G}^\ast, \wt{G}^\ast]$ such that $\chi_{\wt{s}}\in\Irr(\wt{G})$ is irreducible on restriction to $G$ and with distinct $|\cent{\wt G^\ast}{\wt s}|_{3'}$. Note that the centralizer structures of semisimple elements are described in \cite[Prop.~(1A)]{FS82}.

 If $q\geq 9$, then there is some $\zeta\in C_{q-\epsilon}$ with $|\zeta|>2$. Let $\wt s_1$ be the element $\diag(\zeta, \zeta^{-1}, 1,\cdots,1)$ in $[\wt{G}^\ast, \wt{G}^\ast]$. Then $|\cent{\wt{G}^\ast}{\wt{s}_1}|=(q-\epsilon)^2\times |\GL_{n-2}^\epsilon(q)|$ and the character $\chi_{\wt{s}_1}$ is irreducible on restriction to $G$ since $\wt{s}_1$ is not $\wt{G}^\ast$-conjugate to any $\wt{s}_1z$ for $z\in \zent{\wt{G}^\ast}$. Similarly, letting $\wt{s}_2:=\diag(\zeta, \zeta, \zeta^{-2},1,\cdots 1)$, we have $\chi_{\wt{s}_2}$ is irreducible on restriction to $G$ and $|\cent{\wt{G}^\ast}{\wt{s}_2}|=(q-\epsilon)\times |\GL_2^\epsilon(q)|\times |\GL_{n-3}^\epsilon(q)|$.
 Taking $\wt{s}_3$ to have eigenvalues $\{\xi, \xi^{-1}, 1,\cdots,1\}$ with $|\xi|>2$ and $|\xi|\mid (q+\epsilon)$, we again have $\chi_{\wt{s}_3}$ is also irreducible on $G$ and now $|\cent{\wt{G}^\ast}{\wt{s}_3}|=(q^2-1)\times |\GL_{n-2}^\epsilon(q)|$. Hence (b) and (c) hold.
 
 Further, choosing $\zeta$ or $\xi$ to be a fourth root of unity $i$, depending on whether $q\equiv \epsilon\pmod 4$ or $q\equiv -\epsilon\pmod 4$, the corresponding character $\chi_{\wt{s}}=\chi_{\wt{s}_1}$, resp $\chi_{\wt{s}_3}$, above is further invariant under field automorphisms, as $F_3$ stabilizes the class of $\wt{s}$. Then this character is $X$-invariant. Further, we see that this class is even $\mathcal{D}$-invariant, so $\chi_{\wt{s}}$ extends to $\Aut(S)_{\chi_{\wt{s}}}=\Aut(S)$ by \cite[Prop.~3.4 and proof of Lem.~2.13]{Spaeth12}, giving (a2) when combined with Lemma \ref{lem:Tlemma}. (Indeed,  recall that $\zent{G}$ is a nonexceptional Schur multiplier for $S$. Further, the factor set constructed in \cite[Lem.~2.13]{Spaeth12} is trivial
in this case since $\chi_{\wt{s}}$ is $\mathcal{D}$-invariant.) Considering $\wt{s}_3$ to have nontrivial eigenvalues $\{i, i, -i, -i\}$ when $n\geq 4$, we obtain a second $X$-invariant character not $\Aut(S)$-conjugate to $\chi_{\wt{s}}$, yielding further (a1) when $n\geq 4$.
 
 Then finally, we may assume $n\geq 5$ and $q=3$. Here (a1) is satisfied in exactly the same way by the restriction of the same $\chi_{\wt{s}}$ and $\chi_{\wt{s}_3}$. We  obtain (b) and (c) in this case by using $\wt{s}$ from above, together with  $\wt{s}_3$ and $\wt{s}_4$, where $\wt{s}_4$ has  nontrivial eigenvalues $\{-1, -1\}$.  
\end{proof}

 \begin{lem}\label{lem:typeBCDdefining}
Let $q$ be a power of $3$ and let $S$ be a simple group $\PSp_{2n}( q)$ with $n\geq 2$, $\POmega_{2n+1}(q)$ with $n\geq 3$, or $\POmega_{2n}^\pm(q)$ with $n\geq 4$. Let $X\in\Syl_3(\Aut(S))$. Then there exist at least four $X$-invariant characters with distinct degrees in $\irra{3'}{B_0(S)}$.    
In particular, Theorem \ref{thm:conditionsonsimples} holds for $S$ with (a1).
\end{lem}
\begin{proof}

Again, we have 
$\Irr_{3'}(B_0(S))=\Irr_{3'}(S)$. The case $S= \type{D}_4(q)=\POmega_{8}^+(q)$ is covered in Proposition \ref{prop:defininginitial}, so we assume $S$ is not this group. Then note that $XS/S$ must be generated by field automorphisms.
Let $H=\Sp_{2n}(q)$, $\SO_{2n+1}(q)$, or $\SO_{2n}^\pm(q)$, respectively, and write $\Omega:=\textbf{O}^{p'}(H)$, so that $\Omega$ is perfect, $S=\Omega/\zent{\Omega}$, and $H/\Omega$ is a $2$-group. 
If $S=\PSp_4(q)$, we can see the statement directly from the character table constructed in \cite{srinivasan}. Hence, we may assume that $n\geq 3$.

Note that we can naturally embed $\GL_n(q)$ into $H^\ast\cong \SO_{2n+1}(q)$, $\Sp_{2n}(q)$,  or $\SO_{2n}^+(q)$, and can embed $\GL_{n-1}(q)$ into $\SO_{2n}^-(q)$. Then elements of $\SL_n(q)$ will embed as elements of $[H^\ast, H^\ast]$. In particular let $s_1, s_2, s_3$ be  elements  of $H^\ast$ with nontrivial eigenvalues $\{-1, -1, -1, -1\}$,  $\{i, -i, -i, i\}$, and $\{\zeta, \zeta^3, -1, \zeta^{-1}, \zeta^{-3}, -1\}$, where $\zeta\in\FF_{q^2}^\times$ has order $8$. Then these lie in $[H^\ast, H^\ast]$, so that the corresponding semisimple characters of $H$ are trivial on $\zent{H}$. Further, the class of each $s_i$ is stable under $F_3$, and hence $X$, so the $\chi_{s_i}$ are $X$-invarant. For each $i$, let $\chi_i\in\Irr(\Omega\mid\chi_{s_i})$. Since $|H/\Omega|\leq 2$, we see each $\chi_i$ is also $X$-invariant. We also see the $|\cent{{H}^\ast}{\wt{s}_i}|_{3'}$ are distinct from each other, using the centralizer structures described in \cite[(1.13) and discussion after]{FS89}. Further, of these, only $s_1$ or $s_3$ can potentially be conjugate to $s_iz$ for  $1\neq z\in\zent{H}$, but in the case that the restrictions of $\chi_{s_1}$ and/or $\chi_{s_3}$ are not irreducible (and hence split into exactly two characters),
%so $\chi_{s_1}$ is the only of the three that can possibly not restrict irreducibly. Since it also has the largest $|\cent{{H}^\ast}{\wt{s}_i}|_{3'}$, 
we see that $\chi_1, \chi_2, \chi_3$ still have distinct degrees. 
\end{proof}

For a group $X$ with principal $3$-block $B_0$ and $A\leq \Aut(X)$, we will write $k_{A,0}(B_0)$ for the number of distinct $A$-orbits in $\Irr_{3'}(B_0)$.
\begin{prop}\label{prop:defining10}
Let  $S=G/\zent{G}$ be a simple group of Lie type such that $G=\bG^F$ for $\bG$  a simple simply connected algebraic group and $F\colon \bG\rightarrow\bG$ a Frobenius endomorphism defining $G$ over $\FF_q$ with $q=3^a$. Assume that $S$ is not isomorphic to an alternating group or one of the groups covered in Proposition \ref{prop:spor}.   Then $k_{\Aut(S),0}(B_0(S))\geq 10$.
\end{prop}
\begin{proof}
Write $k:=k_{\Aut(S),0}(B_0(S))$. Arguing as in \cite[Sec.~(2D)]{HSF23},  we have 
\[k\geq \frac{q^r}{|\zent{G}||\Out(S)|}\] where $r$ is the rank of $\bG$. (Indeed, note that this gives a bound for orbits of semisimple, hence $3'$, characters, although this was not needed in loc. cit.)

If $G=\SL_2(q)$, this gives $k\geq \frac{3^a}{4a}$, which is larger than 10 for $a\geq 5$. Note that our assumptions exclude the case $a=1,2$ here. 
If $G=\SL_n^\epsilon(q)$ with $n\geq 3$, we have $|\zent{G}|=(n,q-\epsilon)$ and $|\Out(S)|=2a(n,q-\epsilon)$, so we see from this bound that $k\geq 10$ unless possibly $q=3$ with $n=3,4$. If $S=
\type{D}_n^\pm(q)$ with $n\geq 5$, we have $|\zent{G}|\leq 4$ and $|\Out(S)|\leq 8a$, so $k\geq 3^{an}/(32a)\geq 10$ for $n\geq 6$ or for $a\geq 2$ if $n=5$. For $S=\type{D}_4^\pm(q)$, we instead obtain $k\geq  3^{4a}/(96a)\geq 10$ for $a\geq 2$.  If $S=\type{B}_n(q)$ or $\type{C}_n(q)$ with $n\geq 2$, we have $k\geq  3^{an}/(4a)\geq 10$ for $n\geq 4$ or for $a\geq 2$. In the cases $S=\type{D}_5^\pm(3), \type{D}_4^\pm(3),$ $\type{C}_3(3)$, or $\type{B}_3(3)$, we see from GAP that there are at least 10 distinct $3'$-character degrees.

In the exceptional groups, the same bound for $k$ gives $k\geq 10$, except possibly if $S=\type{G}_2(3)$. In the latter case we can again see the statement explicitly from the character table in GAP.
\end{proof}

\begin{prop}\label{prop:thmAdefining}
Let $S$ be a simple group of Lie type defined in characteristic $3$. Then Theorem \ref{thm:theoremAforalmostsimple} holds for $S$.
\end{prop}
\begin{proof}
We may assume that $S$ is not isomorphic to an alternating group or one of the groups considered in Proposition \ref{prop:spor}. Again, recall that $\irra{3'}{S}=\irra{3'}{B_0(S)}$.

So, we may assume that $S$ and $G$ are as in Proposition \ref{prop:defining10} or that $S=G=\tw{2}\type{G}_2(q^2)$. Then $|Z(G)|$ is not divisible by $3$ and the Sylow $3$-subgroup for $S$ is isomorphic to one of $G$, which is of the form $U=\mathbf{U}^F$ with $\mathbf{U}$ the unipotent radical of an $F$-stable Borel subgroup $\mathbf{B}=\mathbf{U}\bT$ of $\bG$. Here $\bT$ is an $F$-stable maximal torus.

If $S=\tw{2}\type{G}_2(q^2)$ with $q^2=3^{2a+1}$ and $a\geq 1$, then a Sylow $3$-subgroup $U$ of $S$ satisfies $[U:U']=q^2\neq 9$ by \cite[Thm.(2)]{ward}.  From the character table in \cite{ward}, we see that $|\irra{3'}{S}|>18$ and as discussed in Proposition \ref{prop:defininginitial}, there are at least six not $\Aut(S)$-conjugate characters in $\irra{3'}{B_0(S)}$ and at least four of these are $X$-invariant, where $X\in\mathrm{Syl}_3(A)$. Hence $k_0(B_0(A))>9$, using Theorem \ref{thm:thetatheorem} when $|A/S|_3=3$,  and Proposition \ref{prop:Noelialemma} when $|A/S|_3=1$.

In the cases that $S$ is as in Proposition \ref{prop:defining10}, since $G\neq \type{G}_2(3)$, we have by \cite[Lem.~2.2]{dignelehrermichel} that $U/U'$ is isomorphic to the direct product of the additive groups $\FF_{q^{|\omega|}}$, where $\omega$ runs over the orbits of the action induced by $F$ on the simple roots for $\bG$ with respect to $\bT$. We see then that $U/U'$ has size 9 exactly when $(r,q)\in\{(2,3), (1,9)\}$ but $G\neq \type{G}_2(3)$, where $r$ is the rank of $\bG$. These cases are covered in Proposition \ref{prop:spor} or Section \ref{sec:alternating}.
In the remaining cases, we therefore have $[U:U']\neq 9$, and we see from Proposition \ref{prop:defining10} that $k_{\Aut(S),0}(B_0(S))\geq 10$ (and hence $k_0(B_0(A))\geq 10$ if $|A/S|_3=1$). 
This completes the proof  when $|A/S|_3=1$.

Now, assume $|A/S|_3=3$. A Sylow $3$-subgroup $P$ of $A$ can be identified with $U\rtimes C$ where $C$ is a cyclic group of graph-field automorphisms of size $3$ stabilizing $U$. From here, arguing  exactly as in the last two paragraphs of the proof of \cite[Prop.~4.9]{NRSV21} yields $P/P'$ is not $2$-generated, and hence $[P:P']\neq 9$.

From the discussion of $\tw{2}\type{G}_2(q^2)$ above, note that we may now assume that $S$ is as in Proposition \ref{prop:defining10}. If $S$ is not $\PSL_n^\epsilon(q)$, then Proposition \ref{prop:defininginitial} and Lemma \ref{lem:typeBCDdefining} yield at least four $P$-invariant characters in $\irra{3'}{B_0(S)}$ that are not $A$-conjugate. Then Theorem \ref{thm:thetatheorem} yields at last 12 characters in $\irra{3'}{B_0(A)}$.  

We are left to  argue that $k_0(B_0(A))>9$ in the case that $S=\PSL_n^\epsilon(q)$ with $n\geq 2$ and  $S$ is not covered in Proposition \ref{prop:spor}. 
Let $G=\SL_n^\epsilon(q)$, $\wt{S}=\PGL_n^\epsilon(q)$, and $\wt{G}=\GL_n^\epsilon(q)$.
Since $A/S$ has order divisible by $3$, it must be that $A$ contains a field automorphism $\varphi$ of order $3$, so $q=q_0^3$ for some $3$-power $q_0$. 
 Note that by the discussion at the beginning of the section,  it suffices to show that there are three nontrivial classes of semisimple elements $s\in G^\ast$ that are not $\Aut(S)$-conjugate and such that $s\in [G^\ast, G^\ast]$ and $s$ is invariant under $\varphi^\ast$. This will yield four $P$-invariant, not $A$-conjugate  members of $\irra{3'}{B_0(S)}$, so that $k_0({B_0(A)})>9$ by Theorem \ref{thm:thetatheorem}.

First, let $S=\PSL_2(q)$, so we  assume $q_0\geq 9$.  The semisimple elements of $G^\ast$ whose preimages in $\wt{G}^\ast$ have eigenvalues $\{i, -i\}$, $\{\zeta, \zeta^{-1}\}$, and $\{\xi, \xi^{-1}\}$ with $|\zeta|\mid (q_0-1)$ and  $|\xi|\mid (q_0+1)$ but  $|\zeta|, |\xi|>4$ yield three nontrivial classes with the properties discussed above. 
 Next, consider $\PSL_n^\epsilon(q)$ with $n\geq 3$. Let $\zeta\in\FF_{q^2}^\times$ have order $8$. Then the semisimple classes in $G^\ast$ whose preimages in $\wt{G}^\ast$ have nontrivial eigenvalues $\{-1, -1\}$, $\{i, -i\}$, and $\{\zeta, \zeta^{3}, -1\}$ satisfy our properties, and the proof is complete. 
\end{proof}

\subsection{Non-Defining Characteristic}

We now move to the case that $p\neq 3$. That is, $S$ is a simple group of Lie type defined in characteristic distinct from $3$. We will keep the setting of Section \ref{sec:setting}. By an abuse of terminology, we will say that $S$ is defined over $\FF_q$ if either $S=\tw{2}\type{F}_4(q^2)$ or if $F$ is a Frobenius map defining $\bG$ over $\FF_q$.

\begin{prop}\label{prop:4unipsext}
Let $S$ be a simple group of Lie type defined over $\FF_q$ with $3\nmid q$ and $3\mid |S|$. Then $\Irr_{3'}(B_0(\wt{S}))$ contains at least $4$ unipotent characters that extend to $\Aut(S)$ and restrict irreducibly to $\Irr_{3'}(B_0({S}))$, except if
\begin{itemize}
\item $S=\PSL_4^\epsilon(q)$ with $3\mid (q-\epsilon)$ or $\PSp_{4}(2^a)$, in which case there are at least $3$ such characters;
\item $S=\PSL_3^\epsilon(q)$, in which case there are $3$ such characters if $3\mid (q-\epsilon)$ and $2$ if $3\mid (q+\epsilon)$; and
\item $S=\PSL_2(q)$, in which case there are $2$ such characters.
\end{itemize}
Further, these four, respectively $3, 3, 2, 2$, characters can be taken with distinct degrees. In particular, $B_0(\Aut(S))$ contains an extension of each of these characters. 
\end{prop}

\begin{proof}
The first part is proved in \cite[Propositions 3.9--3.11]{RSV20}. The Steinberg and trivial characters are always two of these characters, and we refer the reader to  Tables 2-5 of loc. cit. for the  additional explicitly chosen unipotent characters. In fact,  the nontrivial, non-Steinberg characters listed there have distinct degree, using the formula in \cite[Sec.~4.4]{GM20}. The last statement follows from the structure of $\Aut(S)/S$ and is proven, for example,  in \cite[Lemma 3.3]{GMS}.
\end{proof}

\begin{cor}\label{cor:nondefining}
Theorem \ref{thm:conditionsonsimples} holds for $S$ a group of Lie type defined over $\FF_q$ with $3\nmid q$. Further, (a1) holds unless $S=\PSL_2(q)$ or $3\mid (q+\epsilon)$ and $S=\PSL_3^\epsilon(q)$.
\end{cor}
\begin{proof}
This follows from Proposition \ref{prop:4unipsext} if $S$ is not one of the excluded cases there, and part (a) follows also in those cases. (Note that we achieve (a1) except in the case $S=\PSL_2(q)$ and the case $3\mid (q+\epsilon)$ with $S=\PSL_3^\epsilon(q)$, where we obtain (a2) when combined with Lemma \ref{lem:Tlemma}.)

Now let $S=\PSL_2(q)$ or let  $S=\PSL_3^\epsilon(q)=\SL_3^\epsilon(q)$ with $3\mid (q+\epsilon)$. In this case, Sylow $3$-subgroups of $S$ are abelian, and (c) follows from \cite[Thm.~2.1]{HSV23}.  For (b), it suffices to know that there is at least one non-unipotent member in $\irra{3'}{B_0(S)}=\irr{B_0(S)}$. Let $G=\SL_2(q)$, resp. $\SL_3^\epsilon(q)$. In this case, any semisimple $3$-element of $G^\ast$ lies in $[G^\ast, G^\ast]$ and has connected centralizer, so the corresponding semisimple character of $G$ must be trivial on $\zent{G}$ and lie in $B_0(G)$ by the discussion in Sec. \ref{sec:setting}. 

In the remaining cases, it remains to show (c), so we need to exhibit one additional character in $\irra{3'}{B_0(S)}$ not $\Aut(S)$-conjugate to the three unipotent characters. Namely, it suffices again to know there is a non-unipotent member of  $\irra{3'}{B_0(S)}$. 
If $S=\PSp_{4}(2^a)=\Sp_4(2^a)=G$, then any $3$-element $s$ of $G^\ast$ again yields a semisimple character in $B_0(S)$, for the same reasoning. Choosing $s$ to be in $\zent{P^\ast}$ with $P^\ast\in\Syl_3(G^\ast)$ further gives such a character with $3'$-degree. 

Finally, assume $3\mid (q-\epsilon)$ and that $S=\PSL_4^\epsilon(q)$  or $S=\PSL_3^\epsilon(q)$.  Let $G=\SL_4^\epsilon(q)$, respectively $\SL_3^\epsilon(q)$. Then $G$ has a unique unipotent $3$-block using \cite{FS82}, and hence the semisimple character $\chi_s$ lies in $\irr{B_0(G)}$ for any $3$-element $s\in G^\ast$. In the case $G=\SL_4^\epsilon(q)$, we have $3\nmid [G^\ast: [G^\ast, G^\ast]]$, so taking $s$ to be any element in the center of a Sylow $3$-subgroup of $[G^\ast, G^\ast]$, we again obtain $\chi_s\in\irra{3'}{B_0(G)}$ is trivial on the center. Finally, let $G=\SL_3^\epsilon(q)$. The element $(\omega, \omega^{-1}, 1)$ of $\SL_3^\epsilon(q)$, where $|\omega|=3$, commutes with a Sylow $3$-subgroup of $\GL_3^\epsilon(q)$ up to multiplication by $\zent{\GL_3(q)}$, and hence yields a $3$-element $s$ of $[G^\ast,G^\ast]$ centralizing a Sylow $3$-subgroup of $G^\ast$, again yielding a semisimple character $\chi_s\in\Irr_{3'}(B_0(S))$. 
\end{proof}

We remark that this completes the proof of Theorem \ref{thm:conditionsonsimples}.

\subsection{Sylow Subgroups with $[P:P']=9$ in Non-Defining Characteristic}

Here we discuss when the condition $[P:P']=9$ holds for groups of Lie type defined in characteristic $p\neq 3$. It is worth noting that when $[P:P']=9$, we necessarily have $[P:\Phi(P)]\leq 9$, and $P$ is either cyclic or $2$-generated. It will  also be useful to note that when $G=H\wr K$ is a standard wreath product, we have $G/G'\cong H/H'\times K/K'$, for example from \cite[Prob.~1.6.20]{RobinsonGroupTheory}.

Throughout, for an integer $m$ with $3\nmid m$, we will write $d_3(m)$ for the order of $m$ modulo $3$. If $3\mid m$, we write $m_3$ for the highest power of $3$ dividing $m$.

\subsubsection{$\PSL_2(q)$ and $\PSL_3^\epsilon(q)$}

We begin with the case of $\PSL_2(q)$ and $\PSL_3^\epsilon(q)$.

\begin{prop}\label{prop:SL2sylow}
Let $q$ be a power of a prime $p\neq 3$ and let $S=\PSL_2(q)$ or let $3\mid (q+\epsilon)$ and $S=\PSL^\epsilon_3(q)$. Let $A$ be an almost simple group $S\lhd A\leq \Aut(S)$ with $|A/S|_3\leq 3$ and let $P\in\Syl_3(A)$. Then $[P:P']=9$ if and only if $(q^2-1)_3=9$. Further, if $3\nmid |A/S|$, then  $A=S\cent{A}{P}$.
\end{prop}
\begin{proof}
Note that in this case $3\nmid |\wt{S}/S|$. First suppose that $3\nmid |A/S|$, so we may take $P\in\Syl_3(S)$. Let $q=p^f$ with $f=3^a\cdot m$ and $3\nmid m$. Then $A\leq \wt{S}\rtimes \langle \tau, F'\rangle$ with $F'=F_p^{3^a}$ and $\tau$ a graph automorphism in the case of $\PSL_3(q)$ and trivial otherwise. Here $P$ is cyclic of order $(q^2-1)_3$, and we can identify $P$ with a Sylow $3$-subgroup of a cyclic Sylow $d$-torus of $G=\SL_2(q)$ resp. $\SL_3^\epsilon(q)$, where $d=d_3(q)$, resp $d_3(\epsilon q)=2$. The automorphism $F'$ acts on this cyclic torus by $p^{d\cdot 3^a}$ powers, so we see $F'$ acts trivially on $P$ or by inversion, since we may write $(q^2-1)=\prod_{k\mid m}\Phi_k(p^{2\cdot 3^a})$ and since $3\mid (p^{2\cdot 3^a}-1)$, we have $3\nmid \Phi_k(p^{2\cdot 3^a})$ for any $1\neq k\mid m$. 
In either case, we may find $g\in S$ such that $F'\in g\cent{A}{P}$. Similarly, $\tau$ acts by inversion on $P$ and $\wt{S}/S$ is either trivial or order $2$ acting trivially on the Sylow $d$ torus containing $P$, and we see $A=S\cent{A}{P}$ as claimed. Further, in this case $P$ is abelian, so $[P:P']=9$ if and only if $(q^2-1)_3=9$.

Now let $|A/S|_3= 3$. Then $P\cong P_1\rtimes C_3$, where $C_3=\langle F_{q_0}\rangle$ with $q=q_0^3$ and $P_1\in\Syl_3(S)$. Since $P_1$ is abelian, note then that $P/P'\cong P_1/  [P_1,F_{q_0}] \times \langle F_{q_0}\rangle$. 
Let $(q^2-1)_3=3^b$. Then $(q^2-1)=(q_0^2-1)(q_0^4+q_0^2+1)$ and the second factor is divisible by $3$ exactly once, so $(q_0^2-1)_3=3^{b-1}$. Then we can see directly that $P_1/  [P_1,F_{q_0}]$ has size $3^{b-1}$, so that $[P:P']=9$ exactly when $b=2$.
\end{proof}

We next complete the picture for $S=\PSL_3^\epsilon(q)$.

\begin{prop}\label{prop:SL3sylow}
Let $q$ be a power of a prime and let $\epsilon\in\{\pm1\}$ with $3\mid (q-\epsilon)$. Let $G=\SL_3^\epsilon(q)$,     $S=\PSL_3^\epsilon(q)$, and $\wt{S}=\PGL_3^\epsilon(q)$. 
 \begin{enumerate}
\item Let  $R\in\Syl_3(X)$, where $X\in\{S, \wt{S}, G, G\langle \delta\rangle\}$, where $\delta$ is a diagonal automorphism of order $3$. Then 
 $[R:R']=9$.
 \item If $S\leq A\leq\Aut(S)$ with $|A/S|_3=3$, then $R\in\Syl_3(A)$ satisfies $[ R: R']\neq 9$ if and only if the Sylow $3$-subgroup of $A/S$ induces a field automorphism on $S$.
\item If $S\leq A\leq \Aut(S)$ with $3\nmid |A/S|$ and $R\in\Syl_3(A)$, then $[A:M]\leq 2$, where $M=S\cent{A}{R}$. 
\end{enumerate}

\end{prop}
\begin{proof}

Throughout, we let $P$ denote a Sylow $3$-subgroup of $G$ and $\wt{P}$ a Sylow $3$-subgroup of $\wt{G}:=\GL_3^\epsilon(q)$ with $P=\wt{P}\cap G$. Let $Z=\zent{G}$, so $|Z|=3$. Then $Q:=P/Z= PZ/Z\in\Syl_3(S)$ and $\wt{Q}:=\wt{P}\zent{\wt{G}}/\zent{\wt{G}}\in\Syl_3(\wt{S})$.

(i) 
Let $(q-\epsilon)_3=3^b$. Then $\wt{P}\cong C_{3^b}\wr C_3$ and $\wt P'=\{(x,y,z)\in C_{3^b}^3\mid xyz=1\}$. Then we see $P=\wt{P}'\rtimes C_3$ and $P'=[\wt{P}', C_3]$. Since $P$ is not cyclic, we see $[P:P']\neq 3$, and hence $\wt P'\neq P'$. We claim that there is a subgroup $H\leq P'$ with $|H|=3^{2b-1}$, which will force $[P:P']=9$.

Let $a$ be a generator for $C_{3^b}$ and $c$ a generator for $C_3$, chosen such that $[x,c]=(x_1x_2^{-1}, x_2x_3^{-1}, x_3x_1^{-1})$ for $x=(x_1,x_2,x_3)$ in the base group $C_{3^b}^3$. The elements $g_1:=(a,a^{-1},1)$ and $g_2:=(a,a,a^{-2})$ lie in $\wt{P}'$, and $P'$ contains $h_1:=[g_1,c]=(a^2, a^{-1}, a^{-1})$ and $h_2:=[g_2,c]=(1, a^3,a^{-3})$. But note that $|h_1|=3^b$, $|h_2|=3^{b-1}$, and $\langle h_1\rangle \cap\langle h_2\rangle=1$, so $H=\langle h_1, h_2\rangle$ is the desired subgroup. Hence we have shown that $[P:P']=9$.

Now, we have $[Q:Q']=[P:P'Z]$. But note that $Z=\{(x,x,x)\in C_{3^b}^3\mid x^3=1\}=\zent{P}$. Then $Z\leq P'$ since $Z\cap P'\neq 1$, and we therefore have $[Q:Q']=[P:P']=9$.

 Now write $\wt{Z}_3:=\zent{\wt{G}}\cap \wt{P}=\zent{\wt P}$ and note $\wt{Q}\cong\wt{P}/\wt Z_3$. We have $[\wt{Q}:\wt{Q}']=[\wt{P}:\wt{P}'\wt{Z}_3]=[\wt{P}:\wt{P}']/[\wt{Z}_3:\wt{Z}_3\cap\wt{P}']=[\wt{P}:\wt{P}']/[\wt{Z}_3:Z]=[\wt{P}:\wt{P}']/3^{b-1}=3^{b+1}/3^{b-1}=9$.
Now note that $P\langle\delta\rangle\in\Syl_3(G\langle\delta\rangle)$ and $\wt{Q}\cong P\langle\delta\rangle/Z$. Recalling  that $Z\leq P'\leq (P\langle\delta\rangle)'$, we have from before, $9=[\wt{Q}:\wt{Q}']=[P\langle\delta\rangle:(P\langle\delta\rangle)'Z]=[P\langle\delta\rangle:(P\langle\delta\rangle)']$ as claimed.

(ii)  Now suppose that $|A/S|_3=3$. If $A$ contains $\wt{S}$,  then we are done by (i).  So, we may assume $3\nmid [A\cap\wt{S}:S]$. Here, the statement can be seen using  \cite[Lem.~5.1]{NS23}. Alternatively, we may argue as follows.  A Sylow $3$-subgroup of $A$ is of the form $\hat Q=Q\langle F_{q_0}\rangle=P\langle F_{q_0}\rangle/Z$ or $\hat Q=Q\langle \delta F_{q_0}\rangle=P\langle \delta F_{q_0}\rangle/Z$, where $P$, $Z$, $Q$ and $\delta$ are as before and $q=q_0^3$. 
In the case that  $\hat Q=Q\langle \delta F_{q_0}\rangle$, we can compute that $(P\langle \delta F_{q_0}\rangle)'=\wt{P}'$, giving $[\hat Q:\hat Q']=9$. If instead 
 $\hat Q=Q\rtimes\langle F_{q_0}\rangle$, we can compute that $(P\langle F_{q_0}\rangle)'=P'$, yielding that $[\hat Q:\hat Q']=27$.

(iii) Let $q=p^f$ with $f={3^a\cdot m}$ and $(m,3)=1$ and let $F_0:=F_{p^{3^a}}$. In this case, $|\wt{S}/S|=3$, so since $3\nmid |A/S|$, we have $A\leq S\rtimes \langle \tau, F_0\rangle$, where $\tau$ denotes a graph automorphism if $\epsilon=1$ and $A\leq S\rtimes \langle F_0\rangle$ if $\epsilon=-1$. 
First suppose $\epsilon=1$. We choose $\tau$ to be induced by the transpose-inverse automorphism on $G$. Note that $\langle \tau, F_0\rangle=\langle \tau, \tau F_0\rangle \cong C_2\times C_m$. Let $M=S\cent{A}{Q}$. We compute that $F_0$ centralizes $P$ if $p\equiv 1\pmod 3$ and $\tau F_0$ centralizes $P$ if $p\equiv 2\pmod 3$, yielding the claim in this case. 
Now, if $\epsilon=-1$, then $\langle F_0\rangle\cong C_{2m}$. Here we calculate that $F_0$ centralizes $P$ if $p\equiv 2\pmod 3$ but acts by inversion on $\wt{P}'$ if $p\equiv 1\pmod 3$. But in any case,  $F_0^2$ centralizes $P$, completing the proof.
\end{proof}

\subsubsection{Other Linear and Unitary Groups}

Throughout this subsection, we let $G=\SL_n^\epsilon(q)$, $\wt{G}:=\GL_n^\epsilon(q)$, $S=\PSL_n^\epsilon(q)$, and $\wt{S}:=\PGL_n^\epsilon(q)$. We will write $\wt{P}\in\Syl_3(\wt{G})$, $P\in\Syl_3(G)$ with $P=\wt{P}\cap G$, $\wt{Q}\in\Syl_3(\wt{S})$ with $\wt{Q}=\wt{P}\zent{\wt{G}}/\zent{\wt{G}}$, and $Q\in\Syl_3(S)$ with $Q=\wt{Q}\cap S\cong P\zent{G}/\zent{G}$.

The situation for $\GL_n^\epsilon(q)$  will be useful for the remaining classical groups, so we next discuss this case.

\begin{prop}\label{prop:GLSylow}
Let $n\geq 1$ and let $\wt P$ be a Sylow $3$-subgroup of $\GL_n^\epsilon(q)$. 

If   $3\mid (q-\epsilon)$, then $[\wt P: \wt P']=9$ if and only if one of the following holds:
\begin{itemize}
\item $n=1$ and $(q-\epsilon)_3=9$; or
\item $n\in\{2, 3\}$ and $(q-\epsilon)_3=3$.
\end{itemize}

If $3\mid (q+\epsilon)$, then $[\wt P:\wt P']=9$ if and only if one of the following holds:

\begin{itemize}
\item $n\in\{2,3\}$ and $(q+\epsilon)_3=9$; or
\item $4\leq n\leq 7$ and $(q+\epsilon)_3=3$.
\end{itemize}

\end{prop}

\begin{proof}
First, suppose $3\mid (q-\epsilon)$. Let $n=a_0+3a_1+3^2a_2+\cdots+3^ta_t$ be the $3$-adic decomposition of $n$. Then by \cite{weir}, we have $\wt P\cong \prod_{i=0}^t \wt P_i^{a_i}$, where $\wt P_i\in\Syl_3(\GL_{3^i}^\epsilon(q))$. From this we already see that $P$ is at least $3$-generated if $\sum_{i=0}^t a_i \geq 3$, so we assume $\sum_{i=0}^t a_i\leq 2$. 

Now, we have $\wt P_0\in\Syl_3(C_{q-\epsilon})$ is isomorphic to $C_{3^b}$, where $3^b=(q-\epsilon)_3$. So, we see if $n=1$, we have $[\wt P:\wt P']=|\wt P|=9$ if and only if $(q-\epsilon)_3=9$, giving the first item. 

Again by \cite{weir}, for $i\geq 1$, we have $\wt P_i=\wt P_{i-1}\wr C_3=C_{3^b}\wr C_3\wr C_3\wr\cdots\wr C_3$, with $i$ copies of $C_3$ in this iterated wreath product.  
If $\sum_{i=0}^t a_i=2$, then the condition that $\wt P$ is generated by at most two elements forces that each $\wt P_i$ with $a_i\neq0$ is cyclic. This only happens for $i=0$, and hence this is the case $n=2$. Since the  Sylow $3$-subgroup of $\GL_2^\epsilon(q)$ is of the form $C_{3^b}\times C_{3^b}$, we have $[\wt P:\wt P']=9$ if and only if $b=1$ in this situation.

Hence we may assume $n=3^i$ for some $i\geq 1$, and $\wt P=\wt P_i$. Then by the comment about wreath products at the beginning of the section, we have $\wt P_i/\wt P_i'\cong \wt P_{i-1}/\wt P_{i-1}'\times C_3\cong \wt P_0/\wt P_0' \times C_3^i\cong C_{3^b}\times C_3^i$, and the statement follows.

Now suppose $3\mid (q+\epsilon)$. Then letting $w=\lfloor \frac{n}{2}\rfloor$, we can instead write $w=a_0+3a_1+\cdots+3^ta_t$. The work of \cite{weir} again  shows $P$ is isomorphic to $\prod_{i=0}^t \wt P_i^{a_i}$, where now $\wt P_i\in\Syl_3(\GL_{2\cdot 3^i}^\epsilon(q))$. In this case, each $\wt P_i$ is of the same form as before, where now $3^b=(q+\epsilon)_3$, and the result follows. 
\end{proof}

\begin{prop}\label{prop:SLSylow}
Let $n\geq 3$ and let $P$ be a Sylow $3$-subgroup of $G=\SL_n^\epsilon(q)$ with $3\mid (q-\epsilon)$.   Then $[P:P']=9$ if and only if $n=3$ or if $n=4$ and $(q-\epsilon)_3=3$.
\end{prop}
\begin{proof}
From the discussion in the proof of Proposition \ref{prop:GLSylow}, we have $\wt{P}/\wt{P}'\cong \prod_{i=0}^t (C_{3^b}\times C_3^i)^{a_i}$. So $[\wt{P}:\wt{P}']=3^{b\sum a_i+\sum ia_i}$.

Further, we have $[\wt{P}:P]=3^b$, $P'\leq  \wt{P}'\leq P\leq \wt{P}$, and $[P:P']=3^{-b}[\wt{P}:\wt{P}'][\wt{P}':P']$. Then $[P:P']$ is divisible by $3^m$ with $m={b(\sum a_i-1)+\sum ia_i}$. Note that $m>2$ unless we are in one of the following situations:   $n=3$; $n=4$ and $b=1$; or $n=9$.

In the latter case, we can see explicitly that $P=P_T\rtimes P_W$ with $P_T\cong C_{3^b}^8$ and $P_W\cong C_3\wr C_3$ has a minimal generating set of three elements, and hence   $[P:P']\neq 9$. The case $n=3$ is covered in Proposition \ref{prop:SL3sylow}. Finally, if $n=4$ with $b=1$, then $P$ is isomorphic to the Sylow $3$-subgroup for $\SL_4(4)$, and we see explicitly that $[P:P']=9$  in this case, for example in GAP.
\end{proof}

\begin{lem}
\label{lem:CenterSL}
Let $n\geq 3$ and let $P$ be a Sylow $3$-subgroup of $G=\SL_n^\epsilon(q)$ with $3\mid (q-\epsilon)$.  Let $Z\in\Syl_3(\zent{G})$. Then $Z\leq  \wt{P}'$.
\end{lem}
\begin{proof}
First, assume $n=3^i\geq 3$. For $i=1$, we see explicitly that $Z=\zent{P}$, so that $Z\cap P'\neq 1$, and hence $Z\leq P'\leq \wt{P}'$ since it has size $3$. We now induct on $i$. Write $P_i=\wt{P}_i\cap G$, where $\wt{P}_i\in\Syl_3(\GL_{3^i}^\epsilon(q))$. 
We have $\wt{P}_i'=\{(x_1,x_2,x_3)\in \wt{P}_{i-1}^3 \mid x_1x_2x_3 \in \wt{P}_{i-1}'\}$ by \cite[Lem.~(1.4)]{olsson76}.
Further, we can write $Z=\langle z\rangle $ where $z=\zeta\cdot I_{3^i}=(x,x,x)$ with $x=\zeta\cdot I_{3^{i-1}}$ and $\zeta\in C_{q-\epsilon}$ such that $(\zeta^3)^{3^{i-1}}=\zeta^{3^i}=1$. Note that we have $x\in \wt{P}_{i-1}$ with $x^3\in \zent{\SL_{3^{i-1}}^\epsilon(q)}$. So $x^3\in \wt{P}_{i-1}'$ by induction and we conclude $z\in\wt{P}_i'$.

Now, suppose that $n=a_0+a_13+a_23^2+\cdots +a_t3^t=\sum a_i3^i$ is the $3$-adic expansion of $n$, and let $j$ be the smallest index such that $a_j\neq 0$. Then $Z$ is generated by some $\zeta I_n$ with $|\zeta|= \gcd(3^j, 3^b)=\gcd(n, q-\epsilon)$, and we see $\zeta^{3^i}=1$ for each $i$ such that $a_i\neq 0$. That is, $\zeta I_n\leq \prod_{i=0}^t Z_{3^i}^{a_i}$, where $Z_{3^i}$ is a Sylow $3$-subgroup of $\zent{\SL_{3^i}^\epsilon(q)}$. By above, we have $Z_{3^i}\leq \wt{P}_i'$, so we have $Z\leq \prod (\wt{P}_i')^{a_i} =\wt{P}'$ as desired.
\end{proof}

\begin{cor}\label{prop:PSLSylow}
Let $n\geq 3$ and let $Q$ be a Sylow $3$-subgroup of $S=\PSL_n^\epsilon(q)$ with $3\mid (q-\epsilon)$.   Then $[Q:Q']=9$ if and only if $n=3$ or if $n=4$ and $(q-\epsilon)_3=3$.

\end{cor}
\begin{proof}
Let $G=\SL_n^\epsilon(q)$ and write $Z:=\zent{G}$. Let  $Z_3\in\Syl_3(Z)$. Then  $Q=PZ/Z\cong P/Z_3$ with $P\in\Syl_3(G)$. We have $[Q:Q']=[P/Z_3:P'Z_3/Z_3]=[P:P'Z_3]$. If $Z_3\leq P'$, we are done by Proposition \ref{prop:SLSylow}. Write $n=\sum a_i3^i$ and $(q-\epsilon)_3=3^b$ as before. If $n$ is not a power of $3$, then $\sum a_i\geq 2$ and we have $[Q:Q']\geq 3^{m-b}\geq 3^{\sum ia_i}$, where $m$ is as in the proof of Proposition \ref{prop:SLSylow}. For $n>9$, we have $\sum ia_i\geq 3$. Hence we are left to consider the cases $n=6$ and $n=3^i$.
If $n=3$, we saw in the proof of Lemma \ref{lem:CenterSL} that $Z_3\leq P'$. If $n=6$, then we again have $Z_3\leq P'$, since $P\cong P_3\times P_3$, and we have $Z_3\leq P_{3}'\times P_{3}'\leq P'$.

So let $n=3^i\geq 9$.  By Lemma \ref{lem:CenterSL} we have $P'Z_3\leq \wt{P}'$. Then exactly as in the proof of Proposition \ref{prop:SLSylow}, we see $[P:P'Z_3]$ is divisible by $3^m$, and we reduce to the same cases: $n=3$ or $n=9$.
Since $Z_3\leq P'$ when $n=3$, this case follows from Proposition \ref{prop:SLSylow}. When $n=9$,  we see again that $P/Z_3$ has a minimal generating set of size three, so $[Q:Q']\neq 9$.
\end{proof}

\begin{cor}\label{cor:PSLdiagSylow}
Let $n\geq 3$ and let $S=\PSL_n^\epsilon(q)$ with $3\mid (n, q-\epsilon)$. Let $S\leq A\leq \Aut(S)$ such that $|A/S|_3 = 3$ and $A$ contains a diagonal automorphism of order $3$.  Let $R\in\Syl_3(A)$ with $[R:R']=9$. Then $n=3$.
\end{cor}
\begin{proof}
We may write $R=(P\langle\delta\rangle)/Z\leq \wt{P}/Z$, where $P$ and $\wt{P}$ are Sylow $3$-subgroups of $G$ and $\wt{G}$ as before, $\delta$ is a diagonal automorphism of order $3$, and $Z\in\Syl_3(\zent{G})$. Recall from Lemma \ref{lem:CenterSL} that $Z\leq \wt{P}'$. Further, notice that $(P\langle\delta\rangle)'\leq (P\langle\delta\rangle)'Z\leq \wt{P}'\leq P\leq P\langle\delta\rangle\leq \wt P$, so that $[R:R']=[P\langle\delta\rangle:(P\langle\delta\rangle)'Z]\geq [P\langle\delta\rangle:\wt{P}']=3^{1-b}[\wt{P}:\wt{P}']\geq 3\prod_{i=0}^t 3^{ia_i}$, where the last equality is from the calculations in the proof of Proposition \ref{prop:GLSylow}. Then if $[R:R']=9$, we must have $\sum_i ia_i=1$, and hence $n\in\{3,4,5\}$. Then since $3\mid n$, we must have $n=3$ as claimed.
\end{proof}

\subsubsection{The Remaining Cases}

We now summarize (when combined with Propositions \ref{prop:SL2sylow} and \ref{prop:SL3sylow}) when $[P:P']=9$ holds for simple groups of Lie type in non-defining characteristic.

\begin{prop}\label{prop:SylowSideNondef}
Let $S$ be a simple group of Lie type defined in characteristic $p\neq 3$ and let $P\in\Syl_3(S)$. Assume that $S\not\in\{\PSL_2(q), \PSL_3^\epsilon(q)\}$. Then $[P:P']=9$ if and only if $S$ is one of:
\begin{enumerate}
\item $\PSL_4^\epsilon(q)$ when $3\mid\mid(q-\epsilon)$;
\item $\PSL_n^\epsilon(q)$ with $4\leq n\leq 7$ and $3\mid\mid(q+\epsilon)$;
\item $\PSp_{2n}(q)$ or $\POmega_{2n+1}(q)$ with $n\in\{2,3\}$ and  $3\mid\mid(q^2-1)$;
\item $\POmega_8^{-}(q)$ when $3\mid\mid(q^2-1)$;

\item $\type{G}_2(q)$, $\tw{3}\type{D}_4(q)$, or  $\tw{2}\type{F}_4(q^2)'$. 
\end{enumerate}
In (ii)-(v), we further have $\wt{S}A=\wt S\cent{A}{P}$ if $S\leq A\leq \Aut(S)$ with $3\nmid|A/S|$.
\end{prop}
\begin{proof}

First let $S=\PSL_n^\epsilon(q)$. If $3\mid (q-\epsilon)$, we are done by Proposition \ref{prop:PSLSylow}. So assume $3\mid (q+\epsilon)$. Then $P$ can be identified with a Sylow $3$-subgroup of $\wt{G}=\GL_n^\epsilon(q)$ and the statement about $[P:P']$ is Proposition \ref{prop:GLSylow}.

Now suppose that $S$ is of one of the remaining classical types, so that $S=\PSp_{2n}(q)$ with $n\geq 2$, $\POmega_{2n+1}(q)$ with $n\geq 3$, or $\POmega_{2n}^\pm(q)$ with $n\geq 4$. In these cases, $P$ can be identified with a Sylow $3$-subgroup of the corresponding symplectic or orthogonal group $\Sp_{2n}(q)$, $\SO_{2n+1}(q)$, or $\SO_{2n}^\pm(q)$. Weir \cite{weir} has again described the Sylow subgroups in these cases. In the first two cases, $P$ can be identified with a Sylow $3$-subgroup of $\GL_n(q)$ if $3\mid (q-1)$ and  of $\GL_{2n}(q)$ if $3\mid (q+1)$. The Sylow $3$-subgroups of $\SO_{2n}^\pm(q)$ can be identified with those of either $\SO_{2n+1}(q)$ or $\SO_{2n-1}(q)$. Note that in the latter situation when $n=4$, the order polynomials show that  $P$ is in $\SO_{7}(q)$ in the case $-$ and $\SO_9(q)$ in the case $+$. Together with Proposition \ref{prop:GLSylow}, this now gives the first statement for the classical groups. Arguments just like in Proposition \ref{prop:SL2sylow} and using the explicit structure of $P$ yield that $\wt{S}A=\wt{S}\cent{A}{P}$ in  cases (ii)-(iv).

Now let $S$ be an exceptional group of Lie type, including $\tw{3}\type{D}_4(q)$ and the Ree groups $\tw{2}\type{F}_4(q^2)'$. (Recall that the orders of the Suzuki groups $\tw{2}\type{B}_2(q^2)$ are not divisible by $3$ and that the Ree groups $\tw{2}\type{G}_2(q^2)$ are defined in characteristic $3$.) If $S=\type{G}_2(q)$ and $\epsilon\in\{\pm1\}$ such that $3\mid (q-\epsilon)$, then we have $P\in\Syl_3(S)$ is isomorphic to a Sylow $3$-subgroup of $\SL_3^\epsilon(q)$.  Similarly,
when $S=\tw{3}\type{D}_4(q)$, we see from \cite[Prop.~2.2]{deriziotismichler} that $P$ is an extension of a Sylow $3$-subgroup of $\SL_3^\epsilon(q)$  by a diagonal automorphism of order three, where  $3\mid (q-\epsilon)$. 
Next, let $S=\tw{2}\type{F}_4(q^2)'$, with $q^2=2^{2a+1}$. Note that $3\mid (q^2+1)$. Here by \cite[Main Thm.]{malle91max}, $S$ has a maximal subgroup $\SU_3(q^2):2$ and $P$ is isomorphic to a Sylow $3$-subgroup of $\SU_3(q^2)$. Then in all of these cases, $[P:P']=9$ by Proposition \ref{prop:SL3sylow}. Now, note that in these cases $A/S$ is generated by field automorphisms. Further, recalling from the proof of Proposition \ref{prop:SL3sylow} that  either $F_0$ or $\tau F_0$ centralizes $P$, we obtain that $A=S\cent{A}{P}$ in these cases.

Now consider the remaining cases $S=\type{F}_4(q)$, $\type{E}_6^\epsilon(q)$, $\type{E}_7(q)$, and $\type{E}_8(q)$ and let $G$ be as in our setting \ref{sec:setting}. If $S\neq \type{E}^\epsilon_6(q)$ with $3\mid (q-\epsilon)$ we may identify $P$ with a Sylow $3$-subgroup $\hat P$ of $G$. If $S= \type{E}^\epsilon_6(q)$ with $3\mid (q-\epsilon)$, we have $P= \hat P/Z$, where $\hat P\in\Syl_3(G)$ and $Z=\zent{G}$ has size 3. By \cite[Thm.~4.10.2]{GLS}, we may write $\hat P=P_T\rtimes P_W$, where $P_T$ is homocyclic of exponent $3^b$ with $3^b=(q^2-1)_3$ and $P_W\in\Syl_3(W_d)$, where $W_d=\norm{G}{T}/T$ is the relative Weyl group for a Sylow-$d_3(q)$-torus $T$ of $(\bG, F)$. Further, note that $Z\leq P_T$ in the case of $\type{E}_6^\epsilon(q)$. 
In these cases, the relative Weyl groups are listed in \cite[Table 3.2]{GM20}. In each case, we see that $[P_W:P_W']\geq 9$ and $P_W$ is generated by at least two elements. From this structure, we see $P$ cannot be $2$-generated and $[P:P']>9$.
\end{proof}

Finally, we complete the picture when $|A/S|_3=3$.

\begin{prop}\label{prop:nondefmostA3sylow}
Let $S$ be a simple group of Lie type defined in characteristic distinct from $3$ such that $S$ is not  $\PSL_3^\epsilon(q)$ nor $\PSL_2(q)$. Let $A$ be an almost simple group $S\leq A\leq \Aut(S)$ such that $|A/S|_3=3$. Let $P\in\Syl_3(A)$. Then $[P:P']>9$.
\end{prop}
\begin{proof}
First, note that $S=\tw{2}\type{F}_4(2)'$ has outer automorphism of order $2$, so is not being considered here. Then we may let $S=G/\zent{G}$ where $G=\bG^F$ with $\bG$ simple of simply connected type, and let $d:=d_3(q)\in\{1,2\}$ if $F$ is a Frobenius map. If $S=\tw{2}\type{F}_4(q^2)$ (with $q^2>2$), we make the slight alteration on $d$ according to \cite[3.5.3]{GM20}. Let $P_1\in\Syl_3(S)$ and $\hat P_1\in \Syl_3(G)$ such that $P_1=\hat P_1/Z$ with $Z\in\Syl_3(\zent{G})$.  Then as before, $
\hat P_1$ is of the form $\hat P_1=P_T\rtimes P_W$ by \cite[Thm.~4.10.2]{GLS}, where $P_T\in\Syl_3(T)$ for a Sylow $d$-torus $T$ and $P_W\in\Syl_3(N_G(T)/T)$. Then $P_1=(P_T/Z)\rtimes P_W$.

First, assume that  $P=P_1\rtimes\langle \varphi\rangle$, where $\varphi=F_{q_0}$ is a field automorphism of order $3$ (in this case $q=q_0^3$ for some prime power $q_0$) or $S=\type{D}_4(q)$ and $\varphi\in\{F_{q_0}, \tau, \tau F_{q_0}\}$ where $\tau$ is a triality graph automorphism. 
 Write $\hat P=\hat P_1\rtimes\langle \varphi \rangle$. We can choose the decomposition so that $\varphi$ centralizes $P_W$ and normalizes $P_T$. 
 Then we calculate that $\hat P'\leq P_TP_W'$. In particular, we have $[ P: P']\geq 3[P_W:P_W']$, and we may assume that $P_W$ is cyclic of size $3$. However, since $P_1/Z$ is not cyclic, we still see from the structure above that $P$ cannot be $2$-generated, completing the proof in this case.
 
 Then we are left with the case that $S=\PSL_n^\epsilon(q)$ with $3\mid (n, q-\epsilon)$ or $\type{E}_6^\epsilon(q)$ with $3\mid (q-\epsilon)$.
First, suppose  $S=\PSL_n^\epsilon(q)$. If $3\mid|\wt{S}/S|$ and $q=q_0^3$ for some prime power $q_0$,  then \cite[Thm.~5.1]{NS23} yields that if $[P:P']=9$, then $n=3$. If $3\mid |(A\cap\wt{S})/S|$, then Corollary \ref{cor:PSLdiagSylow} also tells us that $n=3$. Hence in this case, we may assume  that $P=P_1\rtimes\langle F_{q_0}\rangle$, and we are done by the first part of the proof.

Finally, suppose that $S=\type{E}_6^\epsilon(q)$ with $3\mid (q-\epsilon)$. If $P$ is not of the form $P=P_1\rtimes\langle F_{q_0}\rangle$ as above, then $P$ is either a Sylow $3$-subgroup $\wt{P}$ of the adjoint version $\wt{S}=\type{E}_6^\epsilon(q)_{\mathrm{ad}}$, or $P=P_1\langle\delta F_{q_0}\rangle$, where $\delta$ is a diagonal automorphism of order $3$ induced by $\wt{S}/S$. In the latter case, we have $P$ is an index-3 subgroup of $\wt{P}\rtimes\langle F_{q_0}\rangle$. Now, by \cite[Thm.~4.10.2]{GLS}, we also have $\wt{P}=\wt{P}_T\rtimes P_W$, and we note that by the same reasoning as in Proposition \ref{prop:SylowSideNondef}, $|\wt{P}/\wt{P}'|>9$. (See also the proof of \cite[5.2]{NS23}.) Now let $P=P_1\langle\delta F_{q_0}\rangle\leq \wt{P}\langle F_{q_0}\rangle$. Note that $P_1$ has index 9 in $\wt{P}\langle F_{q_0}\rangle$, and therefore $P_1$ contains $(\wt{P}\langle F_{q_0}\rangle)'=\wt{P}'[\wt{P}, \langle F_{q_0}\rangle]$, which contains $P'$. But here, we see similar to before that $P_1/(\wt{P}\langle F_{q_0}\rangle)'$ has a quotient $P_W/P_W'$, which has size 9. But then $[P:P']=3[P_1:P']\geq 27[(\wt{P}\langle F_{q_0}\rangle)':P']$, completing the proof.
\end{proof}

\subsection{Theorem \ref{thm:theoremAforalmostsimple} in Non-Defining Characteristic}

In this section, we let $p\neq 3$ and let $q$ be a power of $p$. With the classification of when $[P:P']=9$ now complete, here we turn our attention to the ``character side" of proving Theorem \ref{thm:theoremAforalmostsimple}. We begin by completing the cases of $\PSL_2(q)$ and $\PSL_3^\epsilon(q)$.

\begin{prop}\label{prop:SL23nondef}
Let $S=\PSL_2(q)$ or $\PSL_3^\epsilon(q)$ with $3\nmid q$. Then Theorem \ref{thm:theoremAforalmostsimple} holds for $S$.
\end{prop}
\begin{proof}
Let $S\leq A\leq\Aut(S)$ such that $|A/S|_3\leq 3$. Let $G:=\SL_2(q)$, resp. $\SL_3^\epsilon(q)$.  

First, suppose that $S=\PSL_2(q)$ or that $3\mid (q+\epsilon)$ and $S=\PSL^\epsilon_3(q)$. Recall that the ``Sylow side" has been discussed in Proposition \ref{prop:SL2sylow}, and we keep the notation from that proof. Since $\Aut(S)=\wt{S}\rtimes\langle \tau, F_p\rangle$ and the only outer automorphisms of order $3$ are field automorphisms,  we can write $A$ as $M\rtimes \langle F_{q_0}^i\rangle$ where $q=q_0^3$, $i\in\{0,1\}$, and $3\nmid |M/S|$. Let $P=P_1\rtimes \langle F_{q_0}^i\rangle$ with $P_1\in\Syl_3(S)$. By Proposition \ref{prop:SL2sylow}, we have $M=S\cent{M}{P_1}$ and $\wt{S}=S\cent{\wt{S}}{P_1}$ since $3\nmid |\wt{S}/S|$, so that $k_0(B_0(M))=k_0(B_0(S))=k_0(B_0(G))$ by the Alperin--Dade correspondence Theorem \ref{alperin-dade} and the fact that $3\nmid|\zent{G}|$.
We see using \cite[Thm.~5.12]{malle19} and \cite[Prop.~2.13]{MO83} that $k_0(B_0(S))\leq 9$ exactly when $|P_1|\leq 9$ in this case.

Then suppose that $ |A/S|_3=3$. The set $\irra{3'}{B_0(S)}$ consists of the unipotent characters $1_S$ and $\mathrm{St}_S$, together with the deflation of the semisimple  characters $\chi_s$ where $s\in G^\ast=\PGL_n^\epsilon(q)$ (for the appropriate $n, \epsilon$) lifts to an element of $\wt{G}^\ast=\GL_n^\epsilon(q)$ with nontrivial eigenvalues $\{\zeta, \zeta^{-1}\}$ with $|\zeta|\mid (q^2-1)_3:=3^b$.
Here we see the characters $1_S, \mathrm{St}_S$, and the deflation of  $\chi_s$ with $|\zeta|$ dividing $(q_0^2-1)_3=3^{b-1}$  are invariant under $F_{q_0}$ and are further the only characters in $\irra{3'}{B_0(S)}$ invariant under $F_{q_0}$.  If $b=2$, this yields only $1_S, \mathrm{St}_S$, and $\chi_s$ with  $|\zeta|=3$, so there are exactly 9 characters in $\irra{3'}{B_0(A)}$, lying above the three corresponding characters in $\irra{3'
}{B_0(M)}$. On the other hand, if $b\geq 3$, we obtain at least a fourth not $\Aut(S)$-conjugate character in $B_0(S)$ using $|\zeta|=9$, giving more than $9$ characters in $\irra{3'}{B_0(A)}$ using Theorem \ref{thm:thetatheorem}. When combined with Proposition \ref{prop:SL2sylow}, this completes the proof in this case.

Now let $S=\PSL_3^\epsilon(q)$ with $3\mid (q-\epsilon)$. We keep the notation and setting of the proof of Proposition \ref{prop:SL3sylow}. The set $\irra{3'}{S}$ and action of $\Aut(S)$ on $\irra{3'}{S}$ is described in \cite[Secs.~3.1, 3.2]{malle08}.  

 We have $|\irra{3'}{S}|=6$, with three of these characters coming from the deflation of restrictions of unipotent characters of $\wt{G}$, and the remaining three coming from the deflation of the three constituents of $\irr{G\mid \wt{\chi}_s}$, where $\wt{\chi}_s$ is  a semisimple character of $\wt{G}$ corresponding to a semisimple element $s$ in $\wt{G}^\ast$ of order $3$. From \cite{FS82}, we see that $\wt{\chi}_s$ and the unipotent characters all lie in $B_0(\wt{G})$, so these six characters all lie in $B_0(S)$ and $k(B_0(S))=6$.

Let $P\in\Syl_3(S)$ and suppose we have $3\nmid |A/S|$. Write $M:=S\cent{A}{P}$ as before. From the considerations in \cite[Secs.~3.1, 3.2]{malle08} and discussion in the proof of Proposition \ref{prop:SL3sylow}, we see that in the cases that $[A:M]=2$, the generator of $A/M$ interchanges two of the characters in $\irr{G\mid \wt{\chi}_s}$ and fixes the remaining four characters in $\irra{3'}{S}$. Using Theorem \ref{alperin-dade}, this yields that $k_0(B_0(A))=6$ if $A=M$ and $k_0(B_0(A))=9$ if $[A:M]=2$.

Now, let $|A/S|_3=3$.  Let $Y$ be a Sylow $3$-subgroup of $A/S$. If $Y$ induces a field automorphism of order $3$ on $S$, then again, we have the three unipotent characters are $Y$-invariant. We further see by adapting the descriptions of the characters in \cite[Sec.~3.1]{malle08} that $F_{q_0}$ will stabilize the characters in   $\irr{G\mid \wt{\chi}_s}$, so they are $Y$-invariant. Since these four characters are not $\Aut(S)$-conjugate, we have $k_0(B_0(A))\geq 12$ by Theorem \ref{thm:thetatheorem}.

 Hence, we may assume that $Y$ induces $\delta$ or $\delta F_{q_0}$ on $S$. Then the three characters in $\irr{G\mid \wt{\chi}_s}$ fuse in $A$, and the only characters in $\irra{3'}{B_0(A)}$ are those lying above the three unipotent characters, which extend to $A$ by \cite[Thms.~2.4, 2.5]{malle08}. Let $\bar{S}:=SY$. We have $\irra{3'}{B_0(\bar S)}$ contains the 9 extensions of the three unipotent characters of $S$. Let $\beta\in\irr{\bar S/S}$ be of order $3$. Note that the graph automorphism $\tau$ commutes with $F_{q_0}$ and inverts $\delta$, and therefore inverts $\beta$.
 Now, let $M:=\bar{S}\cent{A}{P}$ with $P\in\Syl_3(A)$, so  $k_0(B_0(M))=9$ by Theorem \ref{alperin-dade}. We have $[A:M]\leq 2$ by the same proof as in Proposition \ref{prop:SL3sylow}(iii). Then when $|A/M|=2$, we have $A$ fuses  two of the characters of $M$ above each unipotent character, corresponding to the fusion of $\beta$ with $\beta^{-1}$. This yields that  $k_0(B_0(A))=9$, as desired.
\end{proof}

The following follows from Proposition \ref{prop:4unipsext} above.

\begin{cor}\label{cor:nondefmostA3}
Let $S$ be a simple group of Lie type defined in characteristic distinct from $3$ such that $S$ is not  $\PSL_3^\epsilon(q)$ nor $\PSL_2(q)$. Let $A$ be an almost simple group $S\leq A\leq \Aut(S)$ such that $|A/S|$ is divisible by 3. Then $k_0(B_0(A))>9$.
\end{cor}
\begin{proof}
If $S$ is not in the list of exceptions in Proposition \ref{prop:4unipsext}, then we have at least four unipotent characters in $\Irr_{3'}(B_0(S))$ that extend to $\Aut(S)$. Then by Theorem \ref{thm:thetatheorem}, we have at least 12 characters in $\Irr_{3'}(B_0(A))$ lying above these four characters of $S$.

Next, suppose that $S=\PSL_4^\epsilon(q)$ with $3\mid (q-\epsilon)$ or $\PSp_4(2^a)$. Then Proposition \ref{prop:4unipsext} and Theorem \ref{thm:thetatheorem} yield at least 9 characters in $\Irr_{3'}(B_0(A))$ lying above unipotent characters in $S$. It suffices to show that there is at least one additional $3'$-degree character. 
Here the class of  $s$ discussed in the proof of Corollary \ref{cor:nondefining} can be chosen to be invariant under field automorphisms of $3$-power order. 
 Using Lemma \ref{lem:Murai}, this completes the proof since the only outer automorphisms of $3$-power order in this case are field automorphisms.
\end{proof}

\begin{prop}\label{prop:10nondefexcept}
Let $q$ be a power of a prime distinct from $3$ and let $S$ be a simple group of Lie type $\type{F}_4(q)$, $\type{E}_6^\epsilon(q)$, $\type{E}_7(q)$, or $\type{E}_8(q)$. Then $k_{\Aut(S),0}(B_0(S))\geq 10$.
\end{prop}
\begin{proof}
First, we remark that all unipotent characters of $3'$-degree lie in $B_0(S)$, as in \cite[Lem.~3.6]{RSV21}. Then from the character degrees of unipotent characters available in \cite[Section 13]{carter}, together with the knowledge of the action of $\Aut(S)$ on unipotent characters (see \cite[Thm.~2.5]{malle08}), we see there are at least 10 unipotent characters of $3'$-degree of $S$ in different $\Aut(S)$-orbits in these cases, except if $S=\type{E}_6^\epsilon(q)$ with $3\mid (q-\epsilon)$. In the latter case, there are 9 such unipotent characters, and the proof of \cite[Prop.~4.7]{NST18} additionally yields at least one non-unipotent character in $\Irr_{3'}(B_0(S))$, which cannot be $\Aut(S)$-conjugate to a unipotent character.
\end{proof}

\begin{prop}\label{prop:9unips}
Let $q$ be a power of  $p\neq 3$ and let $S$ be one of the simple groups $\PSL_n^\epsilon(q)$ with $n\geq 2$; $\PSp_{2n}(q)$ with $n\geq 2$; $\POmega_{2n+1}(q)$ with $n\geq 3$; or $\POmega_{2n}^\pm(q)$ with $n\geq 4$. Let $B_0(S)$ be the principal $3$-block of $S$. Then there are at least $9$  unipotent characters in $\irra{3'}{B_0(S)}$ in distinct $\Aut(S)$-orbits, unless: 
\begin{itemize}
\item $S=\PSL_n^\epsilon(q)$ with $3\mid (q-\epsilon)$ and $n\leq 5$;
\item $S=\PSL_n^\epsilon(q)$ with $3\mid (q+\epsilon)$ and $n\leq 7$; 
\item $S=\PSp_4(q)$, $\PSp_6(q)$, or $\POmega_7(q)$; or 
\item $S=\POmega_{8}^\pm(q)$.

\end{itemize}

\end{prop}
\begin{proof}
Let $H=\GL_n^\epsilon(q)$, $\Sp_{2n}(q)$, $\SO_{2n+1}(q)$, or $\SO^\pm_{2n}(q)$, respectively, and let $d=d_3(q)$. In the case of $\PSL_n^\epsilon(q)$, further define  $e:=d_3(\epsilon q)$. In the other cases, let  $e:=d_3(q^2)$, so that $e=1$. Write $n=ew+r$ with $0\leq r<e$. Note that since  $e\in\{1,2\}$ and $e=1$ in the cases of symplectic and orthogonal groups, the case $w=1$ only happens for $\PSL_2(q)$ or $\PSL_3^\epsilon(q)$ when $e=2$.

Using the results of \cite{BMM, malle07} and arguing as in part (IV) of the proof of \cite[Prop.~3.2]{HSV23}, we see the number of unipotent characters in $\irra{3'}{B_0(S)}$ and of $\irra{3'}{B_0(H)}$ is at least $|\irra{3'}{W_d}|$, where $W_d$ is the relative Weyl group of a Sylow $d$-torus of $H$. Moreover, from \cite[Sec.~3A]{BMM}, we have $W_d=C_e\wr S_w$ if $S=\PSL_n^\epsilon(q)$, $C_{2}\wr S_w$ if $S=\PSp_{2n}(q)$ or $\POmega_{2n+1}(q)$, and is either $C_2\wr S_w$ or is the index-two subgroup $G(2,2,n)$ in the case $S=\POmega_{2n}^\pm(q)$. 

From this, we see that the number of unipotent characters in $B_0(S)$ of height zero is at least $9$ unless we are in one of the stated exceptions.
Further, by \cite[Thm.~2.5]{malle08}, these are $\Aut(S)$-invariant unless $S=\type{D}_n(q)$ or $\Sp_4(2^a)$, in which case the graph automorphisms permute certain unipotent characters. Except for $\type{D}_4(q)$, the graph automorphisms have order 2. If $S=\type{D}_n(q)$, then the second-to-last paragraph of the proof of \cite[Prop.~3.2]{HSV23} yields that there are at least $2|\irra{3'}{S_n}|$ unipotent characters in $\irra{3'}{B_0(S)}$, which is therefore at least 18 for $n\geq 6$. Then we see we still have at least 9 $\Aut(S)$-orbits of unipotent characters, except possibly if $S=\type{D}_5(q)$.  
However, in this case we  see explicitly that the Weyl group $G(2,2,5)$ has 15 characters of $3'$-degree. Hence there are at least 15 unipotent characters in $\irra{3'}{B_0(S)}$. Since at least two of these (namely, $1_S$ and $\mathrm{St}_S$) are $\Aut(S)$-invariant, this yields again at least 9 $\Aut(S)$-orbits of unipotent characters in this case.
\end{proof}

Finally, we complete the proof of Theorem \ref{thm:theoremAforalmostsimple}.

\begin{thm}
Theorem \ref{thm:theoremAforalmostsimple} holds whenever $S$ is a simple group of Lie type.
\end{thm}
\begin{proof}
By Propositions \ref{prop:thmAdefining} and \ref{prop:SL23nondef}, we may assume $S$ is defined in characteristic $p\neq 3$ and that $S$ is not one of $\PSL_2(q)$ or $\PSL_3^\epsilon(q)$. 
If $|A/S|_3=3$, then the result now follows from Proposition \ref{prop:nondefmostA3sylow} and Corollary \ref{cor:nondefmostA3}.
Hence, we may further assume that $S\leq A\leq\Aut(S)$ with $3\nmid |A/S|$.

First, suppose $S$ is of exceptional type. Since $3$ does not divide the order of $\tw{2}\type{B}_2(q^2)$, note that $S$ is not a Suzuki group.
If $S$ is one of $\tw{2}\type{F}_4(q^2)'$, $\type{G}_2(q)$, $\tw{3}\type{D}_4(q)$, we have $k_0(B_0(S))\leq 9$ using the results of \cite{malle91}, \cite{hissshamash}, and \cite{deriziotismichler}, respectively. Since $A=S\cent{A}{P}$ by Proposition \ref{prop:SylowSideNondef}, we see $k_0(B_0(A))\leq 9$ in these cases.
In the remaining exceptional cases, we have $k_0(B_0(A))\geq 10$ by Proposition \ref{prop:10nondefexcept}, and hence we are done when combined with Proposition \ref{prop:SylowSideNondef}.

Now let $S$ be of classical type such that $S$ is not one of the listed exceptions in Proposition \ref{prop:9unips}. Then there are at least 9   orbits under $A$ in $\irra{3'}{B_0(S)}$ from unipotent characters. It suffices in these cases to show there is at least one non-unipotent member of $\irra{3'}{B_0(S)}$, as such a character cannot be $\Aut(S)$-conjugate to a unipotent character.  Here \cite[Lem.~3.3 and Thm.~3.5]{GSV} and their proofs yield a character $\chi\in\irra{3'}{S}$ that is the deflation of a semisimple character $\chi_s\in\irr{G}$, for $s$ a semisimple $3$-element of $G^\ast$.  Then by \cite[Thm.~21.13]{CE04}, we have $\chi_s\in\irra{3'}{B_0(G)}$, so $\chi\in\irra{3'}{B_0(S)}$.

We now deal with the various exceptions pointed out in Proposition \ref{prop:9unips}. The distribution of the unipotent characters of the relevant groups into blocks can be determined by their combinatorial data by \cite{FS82, FS89}. We note that the centralizer structures of semisimple elements in relevant classical groups are also given in loc. cit.     Throughout, we will let $\omega\in\FF_{q^2}^\times$ denote an element of order $3$. We will also let $F_0$ denote the Frobenius map $F_p^{3^a}$, where $q=p^f$ with $f=3^am$ and $3\nmid m$. 

Suppose  $S=\PSL_n^\epsilon(q)$ for some $n\geq 4$ in the list. We will write $G=\SL_n^\epsilon(q)$, $\wt{S}=\PGL_n^\epsilon(q)$, and $\wt{G}=\GL_n^\epsilon(q)$ as before. In each of the relevant cases, we have $3\nmid |\zent{G}|$, so that $\Irr_{3'}(B_0(S))=\Irr_{3'}(B_0(G))$. In this case, we further recall that by \cite[Thm.~2.4, 2.5]{malle08}, every unipotent character extends to $\Aut(S)$. Further, in this case $A\leq\wt{S}\rtimes \langle \tau, F_0\rangle$, where $\tau$ denotes a graph automorphism. (When $\epsilon=-1$, $\tau$ may be identified with the element $F_0^m$ of $\langle F_0\rangle$ of order $2$.) Let $F_0'$ denote $F_0$ if $p\equiv \epsilon \pmod 3$; $\tau F_0$ if $p\equiv -1\pmod 3$ and $\epsilon=1$, and $F_0^2$ if $p\equiv 1\pmod 3$ and $\epsilon=-1$. Then $F_0'$ has index two in the abelian group $\langle \tau, F_0\rangle$, and we write $\tau'$ for a representative of the nontrivial coset. 

First, let $S=\PSL_4^\epsilon(q)$ with $3\mid (q-\epsilon)$.  Here we see that $B_0(S)$ contains all unipotent characters, and hence $B_0(S)$ is the unique unipotent block by \cite[Thm.~9.12]{CE04}. Hence $\irra{3'}{B_0(S)}$ contains three unipotent characters. 
We have $\wt{S}/S$ and $\zent{G}$ are cyclic groups of order prime to $3$, so $\irr{B_0(S)}=\irr{B_0(G)}$ and Lemma \ref{lem:tildeScoprime}  applies to any $\wt{S}$-invariant character in $\irr{B_0(S)}$.
 Let $s'\in \wt{G}^\ast$ be the element $\mathrm{diag}(\omega, \omega, \omega, 1)$ and let $s$ be its image in $G^\ast$ under the natural projection map. Then the class of $s$ and of $s^{-1}$ yield Lusztig series of $G$ that are stable under $\wt{S}$ and $F_0'$ and are interchanged by $\tau'$. Further, since $\cent{\wt{G}^\ast}{s'}=\cent{\wt{G}^\ast}{(s')^{-1}}\cong \GL_3^\epsilon(q)\times C_{q-\epsilon}$ and the three unipotent characters of $\GL_3^\epsilon(q)$ have (distinct) degrees prime to $3$, we see the six characters in these series all lie in $\irra{3'}{B_0(G)}=\irra{3'}{B_0(S)}$. Now, write $M:=\wt{S}\cent{A}{P}$. Then exactly as in Proposition \ref{prop:SL3sylow}(iii), we have $[\wt{S}A:M]\leq 2$, with the nontrivial coset (when it exists) acting on $M$ via $\tau'$. 
 If $(q-\epsilon)_3=3 $, then the classes of $s$ and $s^{-1}$ are the only classes of $3$-elements in $G^\ast$ that yield Lusztig series with $3'$-degree characters. (This can be verified by using \cite[Thm.~5.12]{malle19} and \cite[Prop.~2.13]{MO83} to obtain $k_0(B_0(S))$.) Hence in this case, we have $k_0(B_0(A))= 9$ using Lemma \ref{lem:tildeScoprime}, Theorem \ref{alperin-dade}, and the above analysis about $\tau'$ on these characters.
Now suppose that $(q-\epsilon)_3=3^b\geq 9$.  In this case, we obtain three additional characters in $\irra{3'}{B_0(G)}$, lying in the Lusztig series corresponding to the image in $G^\ast$ of the class representative $(\zeta, \zeta,\zeta,\zeta^{-3})\in\wt{G}^\ast$, where $|\zeta|=3^b$. Since these characters cannot be $\Aut(S)$-conjugate to any of the other 9 characters discussed, we see $k_0(B_0(A))>9$ in this case. 

Now, in the remaining cases, we could make similar explicit arguments. Instead, we will use \cite[Thm.~5.12]{malle19} and \cite[Prop.~2.13]{MO83} to obtain $k_0(B_0(S))$, and apply Proposition \ref{prop:Noelialemma}. In each case, we have at least five unipotent characters in $\Irr_{3'}(B_0(S))$, yielding 5 characters that extend to $A$. Here \cite[Prop.~2.13]{MO83} gives us $k_0(B_0(\wt{G}))$. Then by \cite[Thm.~5.12]{malle19} we have $k_0(B_0(G))=k_0(B_0(\wt{G}))/3^b$ if $3\mid (q-\epsilon)$. If $3\mid (q+\epsilon)$, we instead see that $k_0(B_0(S))=k_0(B_0(\wt{S}))=k_0(B_0(\wt{G}))$ since $3\nmid |\zent{\wt{G}}|$ and using Lemma \ref{lem:tildeScoprime}, since the characters in series indexed by $3$-elements of $\wt{G}^\ast$ must restrict irreducibly to $G$. From this (and using Proposition \ref{prop:SylowSideNondef} and Theorem \ref{alperin-dade}), we obtain that $k_0(B_0(S))=k_0(B_0(A))\leq 9$ whenever $[P:P']=9$, and further that $k_0(B_0(S))>18$ whenever $[P:P']>9$, unless $S=\PSL_6^\epsilon(q)$ or $\PSL_7^\epsilon(q)$ with $(q+\epsilon)_3=9$. Then we are done by Proposition \ref{prop:Noelialemma} unless we are in the  latter case. 

So now suppose $S=\PSL_6^\epsilon(q)$ or $\PSL_7^\epsilon(q)$ with $(q+\epsilon)_3=9$. We see that $B_0(G)$ is the unique unipotent block of maximal defect and contains six unipotent characters of $3'$-degree. 
Recall that $P$ can be identified with a Sylow $3$-subgroup of $\GL_6^\epsilon(q)$ in this case, as can a Sylow $3$-subgroup of $G^\ast$.  Let $s$ be a semisimple element of $G^\ast$ corresponding to an element whose pullback to $\wt{G}^\ast$ has nontrivial eigenvalues $\{\omega, \omega, \omega, \omega^{-1}, \omega^{-1}, \omega^{-1}\}$.  Then $s$ centralizes a Sylow $3$-subgroup.
Since $\cent{\wt{G}^\ast}{s}\cong \GL_3(q^2)\times V$ with $V$ a torus, this series contains three characters of distinct $3'$-degrees. Further, the class of $s$ is invariant under $\Aut(S)$, so these characters are $\Aut(S)$-invariant. 
We then have additional characters in $\irra{3'}{B_0(G)}$ in a series corresponding to eigenvalues $\{\zeta I_3, \zeta^{-1} I_3\}$ with $|\zeta|=9$, yielding  characters not $\Aut(S)$-conjugate to the 9 discussed above, and hence $k_0(B_0(A))>9$.

Together with Proposition \ref{prop:SylowSideNondef}, we see the proof is complete when $S$ is one of the groups $\PSL_n^\epsilon(q)$.

Now let $S=\PSp_4(q)$  and $G=\Sp_4(q)$.  In this case there is again a unique unipotent block of maximal defect, and $\irra{3'}{B_0(S)}=\irra{3'}{B_0(G)}$ contains 5 unipotent characters, which lie in distinct $\Aut(S)$-orbits unless $q=2^a$, in which case they lie in 4 distinct orbits. Since a non-unipotent character cannot be $\Aut(S)$-conjugate to a unipotent character, we will have at least 5 $A$-orbits once we know that $k_0(B_0(S))>4$. Here arguing like above, but now using \cite[Thm.~5.17]{malle19} to obtain $k_0(B_0(G))$, yields $k_0(B_0(A))\leq 9$ if $[P:P]'=9$ and $k_0(B_0(A))> 18$  if $[P:P']>9$, yielding the statement using Proposition \ref{prop:Noelialemma}.

Next, consider the case $S=\PSp_6(q)$. Again we have $\irra{3'
}{B_0(S)}=\irra{3'
}{B_0(G)}$ where $G=\Sp_6(q)$. This case is very similar to that of $\PSL_6^\epsilon(q)$ with $3\mid (q+\epsilon)$ above, again using \cite[Thm.~5.17]{malle19} to obtain $k_0(B_0(S))$. We again obtain $k_0(B_0(A))\leq 9$ if $(q^2-1)_3=3$, $k_0(B_0(S))=18$ if $(q^2-1)_3=9$, and $k_0(B_0(S))>18$ if $(q^2-1)_3>9$. 
If $(q^2-1)_3= 9$, we may argue similar to the case $\PSL_6^\epsilon(q)$ with $3\mid (q+\epsilon)$, now using semisimple elements in $G^\ast=\SO_7(q)$ with eigenvalues $\{\omega, \omega, \omega, \omega^{-1}, \omega^{-1}, \omega^{-1},1\}$ and $\{\zeta, \zeta, \zeta, \zeta^{-1}, \zeta^{-1}, \zeta^{-1},1\}$, whose series yield three characters each in $\irra{3'}{B_0(S)}$, now yielding six $\Aut(S)$-orbits in addition to the six from unipotent characters. Hence again $k_0(B_0(A))>9$.

Now let $S=\POmega_7(q)$. Here note that an order-2 diagonal automorphism is induced by the action of $H:=\SO_7(q)$, if $q$ is odd.  Arguing exactly as in the case of $\Sp_6(q)$ above but with corresponding semisimple elements of $H^\ast=\Sp_6(q)$, we obtain that there are more than 9 $A$-orbits in $\irra{3'}{B_0(H)}$ when $(q^2-1)_3\geq 9$, and we further can argue similarly to above to see $k_0(B_0(A))<10$ when $3\mid\mid(q^2-1)$. Their construction further shows that these 10 characters restrict to distinct $A$-orbits as well in $\irra{3'}{B_0(S)}$, so that $k_0(B_0(A))>9$ when $(q^2-1)_3\geq 9$.

Finally, consider the cases $S=\POmega_8^+(q)$ or $S=\POmega_8^-(q)$ and let $H$  be the corresponding $\SO_8^\pm(q)$ and $\Omega=[H,H]$ so that $S=\Omega/\zent{\Omega}$. In these cases, there are six unipotent characters in $\irra{3'}{B_0(S)}$, and again note that $\irra{3'}{B_0(\Omega)}=\irra{3'}{B_0(S)}$ since $3\nmid |\zent{\Omega}|$. The unipotent characters moved by the graph automorphisms (see \cite[Thm.~2.5]{malle08}) have degree divisible by $3$, so the six unipotent characters in $\irra{3'}{B_0(S)}$ are $A$-invariant. 

If $H=\SO_8^+(q)$, we may consider the  classes in $H^\ast$ of the semisimple elements $s_1$ with eigenvalues $\{\omega, \omega^{-1}, I_6\}$ and $s_2$ with eigenvalues $\{\omega\cdot I_3, \omega^{-1}\cdot I_3, 1,1\}$. These each give series of at least three characters with $3'$-degree. (Note that $\cent{H^\ast}{s_i}$ is of the form $H^\ast\cap(\GL_{n}^\epsilon(q)\times \operatorname{O}_{m}^\epsilon(q))$ where $m$ is the multiplicity of the eigenvalue $1$ and $n$ is the multiplicity of $\omega$, since $n$ is odd.) These two classes cannot be $\Aut(S)$-conjugate and the characters restrict to not-$\Aut(S)$-conjugate characters in $\irra{3'}{B_0(\Omega)}=\irra{3'}{B_0(S)}$ as well since $s_1$ cannot be $\Aut(S)$-conjugate to $s_2z$ for any nontrivial element $z$ in the $2$-group $\zent{H^\ast}$. This gives at least 12 $\Aut(S)$-orbits in $\irra{3'}{B_0(S)}$, so $k_0(B_0(A))>9$.

Finally, let $H=\SO_8^-(q)$, and note that $\cent{H^\ast}{s_i}$, with $s_i$ as above, is of the form $H^\ast\cap(\GL_{n}^\epsilon(q)\times \operatorname{O}_{m}^{-\epsilon}(q))$, where $m$ is the multiplicity of the eigenvalue $1$ and $n$ is the multiplicity of $\omega$, since $n$ is odd. In this case, we see $s_1$ above yields characters of degree divisible by $3$, but $s_2$ still yields three not-$\Aut(S)$-conjugate characters in $\irra{3'}{B_0(H)}$, and hence their restrictions yield three not-$\Aut(S)$-conjugate characters in $\irra{3'}{B_0(S)}$ as before. If $(q^2-1)_3\geq 9$, we may in addition consider the class corresponding to $\{\zeta\cdot I_3, \zeta^{-1}\cdot I_3, 1, 1\}$ with $|\zeta|=(q^2-1)_3$ to obtain three additional $\Aut(S)$-orbits of characters in $\irra{3'}{B_0(S)}$. Then it suffices to show that $k_0(B_0(A))\leq 9$ if $(q^2-1)_3=3$.
Recall that the Sylow $3$-subgroup of $S$ or of $H\cong H^\ast$ can be identified with one of $\SO_7(q)$, and hence of $\GL_3(q)$ or $\GL_6(q)$. From here, the proof is almost exactly as in the case of $\PSL_6^\epsilon(q)$ with $(q+\epsilon)_3=3$ above.
\end{proof}

\section{Proof of Theorem A}

We divide the proof in three parts. We first analyse the character theory of groups whose Sylow 3-subgroups have abelianization of size 9, relying mostly on \cite{NS23}.

\begin{thm}\label{thmA:P/P'=9}
  Let $G$ be a finite group and $P\in \syl 3 G$. Write $B_0=B_0(G)$ to denote the principal $3$-block of $G$. If $[P:P']=9$, then $k_0(B_0)\in \{ 6, 9\}$.
\end{thm}
\begin{proof} We will sometimes write $p=3$.
If $P$ is abelian then $P\cong C_9, C_3\times C_3$. In either case $k_0(B_0)=k(B_0)\in \{ 6, 9\}$ by work of Dade and of Sambale (see \cite[Thm.~8.6 and Cor.~8.11]{Sambale14}, for instance). 
We may thus assume that $P$ is not abelian. In particular $P'=\Phi(P)$.

\smallskip

Arguing by induction on $|G|$ we may assume that ${\bf O}_{3'}(G)=1$ by \cite[Thm.~9.9.(c)]{Navarro98}. We can also assume that $\norm G P <G$, as otherwise the group $G$ would be $p$-solvable and by Fong's Theorem (\cite[Thm.~10.20]{Navarro98}) $k_0(B_0)=k_0(B_0(G/P'))=k(G/P')$. Note that the group $G/P'$ has abelian Sylow $3$-subgroup.

\smallskip

%%%%%%%%%%%%%%%%% O_p(G)=1 %%%%%%%%%%%%%%%%

We want to show that
also ${\bf O}_3(G)$ may be assumed to be trivial.
Otherwise, let $N$ be a minimal normal subgroup of $G$ contained in ${\bf O}_3(G)$. Then %$N\subseteq \cent G N$ is an elementary abelian group and 
$N<P$.
Notice that $(P/N)'=P'N/N=\Phi(P)N/N=\Phi(P/N)$. As $[P/N: P'N/N]$ divides $[P:P']=9$ and $P/N>1$ we conclude that either $P'\leq N$ and $|P/N|=3$ or $N\leq P'= \Phi(P)$.

In the case where $|P/N|=3$ then $N={\bf O}_3(G)$ and $P\cap P^g=N$ whenever $g \in G\setminus \norm G P$. By An--Eaton (see \cite[Cor.~1.6]{AE05} or \cite[Thm.~7.17]{Sambale14}) $k_0(B_0(\norm G P))=k_0(B_0)$ and we are done by induction.
In the case where $N\leq P'$, by induction $k_0(B_0(G/N))\in \{6, 9\}$. We claim that $k_0(B_0(G/N))=k_0(B_0)$. Let $\chi \in{\rm Irr}_0(B_0)$. Then $\chi_P$ contains some linear character $\lambda$ of $P$. Since $\lambda_N=1_N$, we conclude that $\chi$ lies over $1_N$ and $N \subseteq \ker \chi$. By Theorem \ref{thm:NRSV}, $\chi \in {\rm Irr}_0(B_0(G/N))$.

\smallskip

We have shown that if $N$ is a minimal normal subgroup of $G$, then $N$ is semisimple of order divisible by $p$. If $G$ is almost simple, write $S={\rm Soc}(G)$. Recall that $P'=\Phi(P)$. By \cite[Lem.~2.2]{NRSV21} we have that either $[G:S]_3=1$ or $[G:S]_3=3$ and we are done by Theorem \ref{thm:theoremAforalmostsimple}. So we may assume that $G$ is not almost simple. In this case,  \cite[Thm.~3.2 and 3.3]{NS23} yields $k_0(B_0)\leq 9$. As we know that $3$ divides $k_0(B_0)$ (by \cite[Lem.~2.2]{NST18}, for instance) and $k_0(B_0)>3$ by \cite[Thm.~C]{NST18}, we are done in this case.
\end{proof}

We continue by analyzing the structure of  $P/P'$ for groups $G$ satisfying $k_0(B_0(G))=6$ and $P \in \syl 3 G$.

\begin{thm}\label{thmA:case6} Let $G$ be a finite group with $k_0(B_0(G))=6$, where
$B_0$ is the principal $3$-block. Then $[P:P']=9$ where $P\in \syl 3 G$.
\end{thm}
\begin{proof} Write $p=3$ and $B_0=B_0(G)$. We proceed by induction on $|G|$.

\smallskip

  Working by induction, we may assume that ${\rm\textbf{O}}_{p'}(G)=1$ by \cite[Thm.~9.9.(c)]{Navarro98}. We may also assume that $P$ is not normal in $G$. Indeed, if $P\triangleleft G$, then $G$ is $p$-solvable and since ${\rm\textbf{O}}_{p'}(G)=1$ we have that ${\rm Irr}_{p'}(B_0(G))={\rm Irr}_{p'}(G)={\rm Irr}(G/P')$. Analysing the structure of the finite groups with an abelian normal Sylow 3-subgroup and 6 conjugacy classes, we would arrive at $[P:P']=9$.

      \smallskip

 Let $N$ be a minimal normal subgroup of $G$. Then $p$ divides the order of $N$.

    \medskip

%----------------------------------------------------
%N abelian
%-----------------------------------------------------

    (1) Suppose first that $N$ is an elementary abelian 3-group. Notice that $P/N>1$ as ${\bf N}_G(P)<G$, so $3 \mid k_0(B_0(G/N))\leq 6$ by Theorem \ref{Landrock}. Write $\bar B_0=B_0(G/N)$. The only options are then $k_0(\bar B_0)\in \{ 3, 6\}$.

    \smallskip

   (1.a) Suppose that $k_0(\bar B_0)=3$, then $|P/N|=3$ by \cite[Thm.~C]{NST18}. Notice that $N={\bf O}_p(G)$ and that if $g \in G\setminus {\bf N}_G(P)$ then $P\cap P^g=N$. By \cite[Cor.~1.5]{AE05} (see also \cite[Thm.~7.17]{Sambale14}), the block $B_0$ satisfies the Alperin--McKay conjecture. In particular, $k_0(B_0({\bf N}_G(P)))=6$ and we are done by induction (as $P$ is not normal in $G$).

 \smallskip

 (1.b) Suppose that $k_0(\bar B_0)=6$. Then $\irra{p'}{B_0(G)}=\irra{p'}{B_0(G/N)}$, so every height zero character in $B_0(G)$ lies over $1_N$. By induction $[P/N: (P/N)']=9$. Notice that $(P/N)'=P'N/N$. We want to show that actually $N\subseteq P'$ and hence $[P:P']=9$. 
 Given any linear character $\lambda \in {\rm Irr}(P)$, we have that $\tau=\lambda_N\in {\rm Irr }(N)$. By Lemma \ref{lem:Murai} there is some $\chi\in {\rm Irr}(B_0)$ of height zero lying over $\tau$. Then $\tau=1_N$. Therefore $N$ is contained in the kernel of every linear character of $P$, giving us 
 $$N\subseteq\bigcap_{\lambda\in{\rm Lin}(P)}{\rm ker}(\lambda)=P',$$ as wanted.

 \bigskip
 
%----------------------------------------------------
% k0(B0(G/N)=1
%-----------------------------------------------------

   (2) Suppose now that $N$ is semisimple, and write $N=S_1\times S_2\times\cdots\times S_t$, where $S_i\cong S$ are nonabelian simple of order divisible by 3. Since $p=3$, we have that $k_0(B_0(G/N))\in\{1,3,6\}$ again by Theorem \ref{Landrock}. 

    \medskip

    (2.a) Suppose first that $k_0(B_0(G/N))=1$. Then 3 does not divide $[ G:N]$ and $P\leq N$. 
    Also, since ${\rm\textbf{O}}_{p'}(G)=1$, we have that $\cent G N=1$.
     By the Frattini argument, we have that $G=N\norm G P$, so $M=N\cent G P$ is normal in $G$. By Lemma \ref{lemma:onlyblockabove} we have that $B_0(G)$ is the only block covering $B_0(M)$. In particular, since $3$ does not divide $[G:M]$, we have that ${\rm Irr}(G/M)\subseteq{\rm Irr}_0(B_0(G))$. More generally ${\rm Irr}(G|\theta)\subseteq {\rm Irr}_0(B_0)$ for every $\theta$ of height zero in the principal block of $M$.  
     If $t=1$ then $G$ is almost simple and we are done by Theorem \ref{thm:theoremAforalmostsimple}. Hence we may assume that $t\geq 2$.
     By Theorem \ref{thm:conditionsonsimples}(b), the set of degrees of ${\rm Irr}_{p'}(B_0(N))$, has size at least 6.
Notice that for every $\psi\in{\rm Irr}_{p'}(B_0(M))$, there is $\chi\in{\rm Irr}_{p'}(B_0(G))$ lying over it and hence the set of character degrees of characters in ${\rm Irr}_{p'}(B_0(M))\setminus\{1_M\}$ has size at most $6-k(G/M)$. By Theorem \ref{alperin-dade} so has the set of character degrees of characters in ${\rm Irr}_{p'}(B_0(N))\setminus\{1_N\}$. These facts force $M=G$. Then again by Theorem \ref{alperin-dade} we have that $6=k_0(B_0(N))=k_0(B_0(S))^t$, absurd as $t\geq 2$.

%----------------------------------------------------
% k0(B0(G/N)=3
%-----------------------------------------------------    

\bigskip

(2.b) Suppose now that $k_0(B_0(G/N))=3$. Then $|PN/N|=3$. We first claim that $\cent G N=1$. Indeed, consider $k_0(B_0(G/\cent G N))$. As $[G:\cent G N]$ is divisible by 3, we have that $k_0(B_0(G/\cent G N))\in \{ 3,6\}$ by Theorem \ref{Landrock}. If $k_0(B_0(G/\cent G N))=3$  then $[G:\cent G N]_3=3$ by \cite[Thm.~C]{NST18} and hence $|N|_3=3$, which implies $|G|_3=9$, and the result follows. Finally, suppose that $k_0(B_0(G/\cent G N))=6$, so all the characters in $\chi\in{\rm Irr}_0(B_0(G))$ lie over $1_{\cent G N}$. Since $[G:N]_3=3$ and 3 divides $|\cent G N|$, we necessarily have that 3 does not divide $[G:N\cent G N]$. Now take $\varphi\in{\rm Irr}_0(B_0(\cent G N))$ nontrivial and consider $1_N\times \varphi\in{\rm Irr}_0(B_0(N\times \cent G N))$, then there is $\chi\in{\rm Irr}_0(B_0(G))$ lying over $\varphi$, a contradiction.

\medskip

 We now claim that $P$ (and hence  $PN/N$) acts transitively on $\{S_1,\ldots, S_t\}$. In particular, $t\in\{1,3\}$ and by Theorem \ref{thm:theoremAforalmostsimple}, that implies $t=3$. 
 Suppose that the action of $P$ on $\{S_1,\ldots, S_t\}$ is not transitive and write $N=R\times T$, where $R$ is the product of one $P$-orbit. 
 Since $p$ divides $k_0(B_0(R))$, then necessarily $p$ divides the size of the set ${\rm Irr}_{p',P}(B_0(R))$ of $P$-invariant characters of height zero in ${\rm Irr}(B_0(R))$.
 As $1_R\in {\rm Irr}_{p',P}(B_0(R))$, there exists $1_R\neq \theta\in{\rm Irr}_{p'}(B_0(R))$ that is $P$-invariant. Analogously, let $1_T\neq \xi\in{\rm Irr}_{p'}(B_0(T))$ be $P$-invariant. We have that $1_R\times\xi$ and $\theta\times \xi$ are  $P$-invariant, so they extend to $PN$ using \cite[Cor.~8.16]{Isaacs76}. By Theorem \ref{thm:thetatheorem}, there are at least 3 characters in ${\rm Irr}(B_0(G))$ of height zero above them. As $1_R\times\xi$ and $\theta\times \xi$ are not $G$-conjugate, this implies that $k_0(B_0(G))\geq 9$, a contradiction. 
  
 \medskip
 
 Therefore $N=S_1\times S_2\times S_3$, $|PN/N|=3$ and $P$ acts transitively on $\{ S_1, S_2, S_3 \}$. 
  We claim that from every character $\varphi$ in $\irra{p'}{B_0(S)}$ we can construct a $P$-invariant character $\theta$ in $\irra{p'}{B_0(N)}$ with $\theta(1)=\varphi(1)^3$. Indeed, if $xN$ generates $PN/N$, then $xN$ seen as an element of ${\rm Out}(N)$ acts on characters of $N$ like $(\alpha_1, \alpha_2, \alpha_3)\sigma$ with $\alpha_i \in {\rm Out}(S)$ and $\sigma \in { S}_3$ of order $3$, satisfying $\alpha_1\alpha_{\sigma(1)}\alpha_{\sigma(2)}=1$. The character $\theta=\varphi\times \varphi^{\alpha_1}\times \varphi^{\alpha_1\alpha_{\sigma(1)}}$ is $P$-invariant whenever $\varphi$ is a character of $S$, and hence it does extend to $PN$ using \cite[Cor.~8.16]{Isaacs76}. 
  Now, by Theorem \ref{thm:thetatheorem}, there are at least $3$ characters in $\irra{p'}{B_0(G)|\theta}$.  
  With this method, and using Theorem \ref{thm:conditionsonsimples}(b), we can construct at least $2$ non-trivial $P$-invariant characters $\theta\in \irra{p'}{B_0(N)}$ with different degrees, but again that is not possible.

\bigskip

%----------------------------------------------------
% k0(B0(G/N)=6
%-----------------------------------------------------

(2.c) Finally suppose that $k_0(B_0(G/N))=6$, thus every irreducible character in ${\rm Irr}_0(B_0(G))$ lies over $1_N$. Since 3 divides $k_0(B_0(N))$, we have that there are at least 3 irreducible characters in ${\rm Irr}_0(B_0(N))$ fixed by $P$. Let $\varphi$ one of them, so $\varphi$ extends to $PN$, using \cite[Cor.~8.16]{Isaacs76}. Then by Lemma \ref{lem:Murai} there exists $\chi\in{\rm Irr}_0(B_0(G))$ lying over $\varphi$, a contradiction.   
\end{proof}

We now finish the proof of Theorem A.

\begin{thm}\label{thmA:case9}
    Let $G$ be a finite group, $p=3$ and $P\in \syl p G$. Suppose that $k_0(B_0(G))=9$, then $[P:P']=9$.
\end{thm}
\begin{proof} We proceed by induction on $|G|$.
Working as in the proof of Theorem \ref{thmA:case6}, we may assume that ${\rm \textbf{O}}_{p'}(G)=1$ and $\norm G P < G$. Write $B_0=B_0(G)$.
    
    \smallskip
    
  (1)  We next want to prove that we may also assume that ${\rm \textbf{O}}_{p}(G)=1$. Otherwise, let 
    $N$ be a minimal normal subgroup of $G$ contained in ${\rm \textbf{O}}_{p}(G)$.
  Notice that $p$ divides $|G/N|$ as otherwise $N=P\triangleleft G$. By Theorem \ref{Landrock},  $3$ divides $k_0(\bar{B_0})=k_0(B_0(G/N))\leq k_0(B_0)=9$. The only possibilities are then $k_0(\bar{B_0})\in \{ 3, 6, 9\}$.

\medskip

 (1.a) Suppose that $k_0(\bar{B_0})=9$. In this case, ${\rm Irr}_0(B_0(G))={\rm Irr}_0(B_0(G/N))$, so every character of height zero in $B_0(G)$ lies over $1_N$.
Moreover, by induction, $[P:P'N]=9$.  We claim that $N\subseteq P'$. Let $\lambda\in{\rm Lin}(P)$. Then $\theta=\lambda_N\in{\rm Irr}(N)$ extends to $\lambda$. 
By Lemma \ref{lem:Murai} we have that there exists $\chi\in{\rm Irr}_0(B_0(G))$ over $\theta$. Hence $\theta=1_N$. Then $$N\subseteq\bigcap_{\lambda\in{\rm Lin}(P)}{\rm ker}(\lambda)=P',$$ as wanted. Therefore $[P:P']=[P:P'N]=9$ and we are done.

\bigskip

 (1.b) Suppose that $k_0(\bar{B_0})=6$. In this case, by Theorem \ref{thmA:case6}, $[P:P'N]=9$. We work to prove that $N\subseteq P'$. 
 Let $\lambda \in {\rm Lin}(P)$. If $\theta=\lambda_N \neq 1_N$, then by Lemma \ref{lem:Murai} some $\chi \in {\rm Irr}_{p'}(B_0)$ lies over $\theta$. 
By Theorem \ref{thm:thetatheorem}, we have that $3$ divides $k_0(B_0|\theta)$. In fact, $k_0(B_0|\theta)=3$ as $k_0(B_0)=9$ and $k_0(\bar B_0)=6$. Moreover $k_0(B_0|\theta)=k_0(B_0(G_\theta)| \theta)$, by the Clifford correspondence and the Third Main Theorem (\cite[Thm.~6.7]{Navarro98}).
By Lemma \ref{lem:extensionprincipalblock}, we have that $\theta$ extends to some $\psi \in {\rm Irr}_0(B_0(G_\theta))$. Then, the equality $k_0(B_0(G_\theta)|\theta)=k_0( B_0(G_\theta/N))$ follows from Gallagher's Theorem and \cite[Lem.~5.1]{NV17}. But
 $k_0(B_0(G_\theta/N))>3$ by \cite[Thm.~C]{NST18} because $[P/N:(P/N)']=9$, a contradiction. 
Hence $\theta=1_N$, and we conclude $N\subseteq P'$, as in (1.a). Then $[P:P']=[P:P'N]=9$ and we are done.

\bigskip

 (1.c) Suppose that $k_0(\bar{B_0})=3$. Then $|P/N|=3$. Recall that $P$ is not normal in $G$, so $N={\rm\textbf{O}}_p(G)$. Let $g\in G\setminus \norm G P$. Then $P\cap P^g=N$. Again, by Corollary 1.5 of \cite{AE05}, the Alperin--McKay conjecture holds for $B_0(G)$. Hence $k_0(B_0({\bf N}_G(P)))=9$. Since $\norm G P <G$, we are done by induction.

\bigskip

(2) Now let $N$ be a minimal normal subgroup of $G$. By (1),  $N$ is a semisimple group of order divisible by $p$. Then $N= S_1\times \cdots \times S_t$, where $S_i\cong S$ are transitively permuted by the action of $G$, and $S$ is a simple non-abelian group of order divisible by $p$. We now analyse the different possibilities for $k_0(B_0(G/N))\in \{ 1, 3, 6, 9\}$.

\smallskip

(2.a) $k_0(B_0(G/N))=1$. In this case $(|G/N|, p)=1$ and $P\subseteq N$.  Since ${\rm\textbf{O}}_{p'}(G)=1$, we have that $\cent G N =1$, so we may assume that $t>1$ by Theorem  \ref{thm:theoremAforalmostsimple}. Consider $M=N{\bf C}_G(P)$. Then $M$ is normal in $G$, by the Frattini argument, and $B_0$ is the only block covering $B_0(M)$ by Lemma \ref{lemma:onlyblockabove}. In particular $k(G/M)\leq 8$. Notice that by Theorem \ref{thm:conditionsonsimples} (b) and using that $t\geq 2$ we obtain that there are at least 6 different character degrees of irreducible characters in ${\rm Irr}_{p'}(B_0(N))$. 
Using Theorem \ref{alperin-dade}, there are at least 6 different character degrees of irreducible characters in ${\rm Irr}_{p'}(B_0(M))$ too. Since every irreducible character in ${\rm Irr}_{p'}(B_0(M))$ lies under an irreducible character in ${\rm Irr}_{p'}(B_0(G))$ we have that the set of character degrees of irreducible characters in ${\rm Irr}_{p'}(B_0(M))\setminus \{ 1_M\}$ is at most $9-k(G/M)$. This, together with the fact that $3$ does not divide $[G:M]$, forces $k(G/M)\in\{1,2,4\}$.

(2.a.i) If $k(G/M)=1$ then $G=M$, then by Theorem \ref{alperin-dade}, $9=k_0(B_0(N))=k_0(B_0(S))^t$. Hence either $t=1$ and we are done by Theorem \ref{thm:theoremAforalmostsimple} or $t=2$ and $k_0(B_0(S))=3$. By \cite[Thm.~C]{NST18}, in the latter case, $|S|_3=3$ and then $|P|=9$ as wanted.

(2.a.ii) If $k(G/M)=2$ then $[G:M]=2$, and then $k_0(B_0(S))^t=k_0(B_0(N))=k_0(B_0(M))\leq 1+2\cdot 7=15$. The above inequalities force $t=2$ and $k_0(B_0(S))=3$. Then $|P|=9$ by \cite[Thm.~C]{NST18}.

(2.a.iii) Finally suppose that $k(G/M)=4$. In this case, since there are at least 5 $G$-orbits in ${\rm Irr}(B_0(M))\setminus\{1_M\}$, we have that every element of ${\rm Irr}_{p'}(B_0(M))\setminus\{1_M\}$ has exactly one element of ${\rm Irr}_{p'}(B_0(G))$ above it (and there are exactly 5 $G$-orbits in ${\rm Irr}(B_0(M))\setminus\{1_M\}$). Since ${\rm Irr}(B_0(G))$ is the unique block covering ${\rm Irr}(B_0(M))$ this means that all the nontrivial elements in ${\rm Irr}_{p'}(B_0(M))$ are fully ramified in their stabilizers. This means that, if $\varphi\in{\rm Irr}_{p'}(B_0(M))$ then $[G_\varphi:M]$ is a square. By analyzing the structure of groups with 4 conjugacy classes we conclude that either $[G:M]=4$ or $[G:M]=10$ and $G_\varphi=M$ for all $\varphi\in{\rm Irr}_{p'}(B_0(M))$. If $[G:M]=4$, then $k_0(B_0(S))^t=k_0(B_0(N))=k_0(B_0(M))\leq 1+4\cdot 5=21$ and then $t=2$ and $k_0(B_0(S))=3$, as wanted. If $[G:M]=10$, we have that $k_0(B_0(S))^t=k_0(B_0(N))=k_0(B_0(M))= 1+10\cdot 5=51$ which forces $t=1$, contradiction.

\smallskip

(2.b) $k_0(B_0(G/N))=9$. In this case every height zero character in the principal block of $G$ contains $N$ in its kernel. But that is absurd. Indeed, the set ${\rm Irr}_0(B_0(N))$ has size divisible by $3$. In particular, also the set of $P$-invariant characters in ${\rm Irr}_0(B_0(N))$ has size divisible by 3. Thus, we can choose a $P$-invariant $1_N\neq \theta \in {\rm Irr}_0(B_0(N))$. By \cite[Cor.~8.16]{Isaacs76}, $\theta$ extends to $PN$.  By  Lemma \ref{lem:Murai}, there would be some $\chi \in {\rm Irr}_0(B_0)$ lying over $\theta$, so not containing $N$ in its kernel. 

\smallskip

We have shown that if $N$ is a minimal normal subgroup of $G$ then, $k_0(B_0(G/N))\in \{ 3, 6\}$.

\smallskip

(2.c) If $k_0(B_0(G/N))=3$, then $[PN:N]=3$ by \cite[Thm.~C]{NST18}. 
We claim that $N$ is the only minimal normal subgroup of $G$. Once this statement is proven, we can assume that ${\bf C}_G(N)=1$. Suppose that $M$ is another minimal normal subgroup of $G$. Since $3$ divides $|M|$ we have that $3$ divides $k_0(B_0(M))$. Write $K=NM\cong N\times M$. Since $[G:N]_3=3$ and $3$ divides $|M|$ we have that $[G:K]$ is not divisible by 3 and $|M|_3=3$. Hence $P\leq K$. 
Let $L=K{\bf C}_G(P)$ which is normal in $G$ by the Frattini argument. By the Alperin--Dade correspondence (Theorem \ref{alperin-dade}), there are the same number of $G$-orbits in ${\rm Irr}_{p'}(B_0(L))$
as in ${\rm Irr}_{p'}(B_0(K))={\rm Irr}_{p'}(B_0(N))\times {\rm Irr}_{p'}(B_0(M))$. If $L=G$, then by Theorem \ref{alperin-dade}, we have that $9=k_0(B_0(G))=k_0(B_0(K))=k_0(B_0(N))k_0(B_0(M))$, forcing $k_0(B_0(N))=k_0(B_0(M))=3$. Then $|N|_3=|M|_3=3$ by \cite[Thm.~C]{NST18}, so $|G|_3=9$ as desired. Hence, we may assume that $L<G$.

\smallskip

By Lemma \ref{lemma:onlyblockabove} we have that $B_0(G)$ is the only block covering $B_0(L)$. Notice that by Theorem \ref{thm:conditionsonsimples}, there are at least three characters in both ${\rm Irr}_{p'}(B_0(N))$ and ${\rm Irr}_{p'}(B_0(M))$ with distinct degrees, namely $\theta_0=1_N$, $\theta_1$ and $\theta_2$ for $N$ and $\varphi_0=1_M$, $\varphi_1$ and $\varphi_2$ for $M$. Notice then that the set
 $$\Omega=\{ \theta_i\times \varphi_j \ | \ i, j \in \{ 0, 1 , 2\} \}\subseteq {\rm Irr}_{p'}(B_0(K))$$
 consists of non-$G$-conjugate characters. In particular, the action of $G$ on ${\rm Irr}_{p'}(B_0(L))$ defines at least 9 orbits, forcing each $\chi \in {\rm Irr}_{p'}(B_0)$ to lie above a different $G$-orbit on $L$.
 This is absurd as $L<G$ and every character in  ${\rm Irr}(G/L)\subseteq {\rm Irr}_{p'}(B_0)$ lies over $1_L$.
 
 \smallskip
 
 Hence the claim is proven and $\cent G N=1$. By Theorem \ref{thm:theoremAforalmostsimple}, we may assume that $t>1$. We claim now that $P$ acts transitively on $\{ S_1, \ldots, S_t\}$. Note that $|PN/N|=3$ so each of these $P$-orbits has length $1$ or $3$.

(2.c.i) Suppose that the action of $P$ on $\{S_1,\ldots, S_t\}$ is not transitive and defines at least one orbit of length $3$. Write $N=R\times T$, where $R$ is the product of the elements in a $P$-orbit of length 3. 
 Since $P$ acts transitively on the simple factors of $R$, we can argue as in the last paragraph of the case (2.b) of  the proof of Theorem \ref{thmA:case6} to obtain 2 nontrivial $P$-invariant characters $\theta,\xi\in{\rm Irr}_{p'}(B_0(R))$ of different degrees. Now since $|{\rm Irr}_{p'}(B_0(T))|$ is divisible by $p$ and $P$ acts on ${\rm Irr}_{p'}(B_0(T))$, there exists a nontrivial $P$-invariant $\psi\in{\rm Irr}_{p'}(B_0(T))$. Notice that $1_N$, $\theta\times 1_T$, $\xi\times 1_T$, $\theta\times\psi\in{\rm Irr}_{p'}(B_0(N))$ are all $P$-invariant  and not $G$-conjugate. By \cite[Cor.~8.16]{Isaacs76}, they extend to $PN$. 
By Theorem \ref{thm:thetatheorem}, there are at least $3$ irreducible characters in ${\rm Irr}_{p'}(B_0(G))$ above each of them, and no character of $G$ lies above $2$ of them. This yields $k_0(B_0(G))>9$,  a contradiction.

\smallskip

(2.c.ii) Suppose now that the action of $P$ on $\{S_1,\ldots, S_t\}$ is trivial. In this case $P$ acts on ${\rm Irr}_{p'}(B_0(S))$. Since $p$ divides $|{\rm Irr}_{p'}(B_0(S))|$ by Theorem \ref{Landrock}, $p$ also divides the number of $P$-invariant irreducible characters in ${\rm Irr}_{p'}(B_0(S))$ and then every $S_i$ contains a nontrivial $P$-invariant irreducible character of $p'$-degree in the principal block. If $t>2$ we can easily construct 3 nontrivial $P$-invariant character in ${\rm Irr}_{p'}(B_0(N))$ that are not $G$-conjugate. As in the (2.c.i) case, we get a contradiction using \cite[Cor.~8.16]{Isaacs76} and Theorem \ref{thm:thetatheorem}.
We are left with the situation in which $t=2$ and $P$ normalizes $S_i$ for $i \in \{ 1, 2\}$. As before, we can derive a contradiction if we are able to construct $3$ nontrivial $P$-invariant characters  ${\rm Irr}_{p'}(B_0(N))$ that are not $G$-conjugate. That is certainly the case if $S$ satisfies the (a1) statement of Theorem \ref{thm:conditionsonsimples}. By Remark \ref{remark:conditions}, we can assume that $S$ satisfies (a2) and that the Sylow $3$-subgroup of ${\rm Out}(S)$ is cyclic and normal.  Notice that ${\rm Out}(S)$ has only $2$ elements of order $3$. Since $|PN/N|=3$, then it is easy to show that $G/N$ (that can be identified with a subgroup of ${\rm Out}(S)\wr C_2$ because $\cent G N=1$)  has only $2$ elements of order $3$. Henceforth $PN/N \triangleleft G/N$ and $PN \triangleleft G$. At this point, we can consider $M=PN \cent G P$. We have that $M$ is normal in $G$ by the Frattini argument and $B_0$ is the only block  covering $B_0(M)$ by Lemma \ref{lemma:onlyblockabove}. Also, by Theorem \ref{alperin-dade} we know that $k_0(B_0(M))=k_0(B_0(PN))$. Notice that the set of $P$-invariant characters in ${\rm Irr}_{p'}(B_0(S_i))$ has size at least $3$ (in our situation it will be exactly $3$ but it is not really important). Hence the set of $P$-invariant characters in ${\rm Irr}_{p'}(B_0(N))$ has size at least $9$, which translates in $k_0(B_0(M))=k_0(B_0(PN))\geq 27$ by \cite[Cor.~9.6]{Navarro98}. We then see that $M<G$. As $G/M$ is a group of order coprime to $3$ and all its characters are contained in $B_0(G/N)$ (if they would lie in a different block of maximal defect of the quotient contained in $B_0(G)$ we would be done just by constructing $2$ nontrivial $P$-invariant characters in ${\rm Irr}_{p'}(B_0(N))$ lying in different $G$-orbits, that we can do), we get that $G/M\cong C_2$. Then $k_0(B_0)\geq 2+26/2\geq 15$, a contradiction.

 \smallskip

Therefore we can assume that $P$ acts transitively on the set of simple direct factors of $N$ and therefore $t=3$. If $|S|_3=3$, then notice that $P$ has a normal subgroup $Q\cong { C}_3\times { C}_3\times { C}_3$, $|P/Q|=3$ and $P$ acts permuting transitively the ${ C}_3$-factors of $Q$.  In particular, $[P:P']=9$, as wanted. So we may assume that $9$ divides $|S|$. By Theorem \ref{thm:conditionsonsimples}(c), notice that $N$ has at least 3 $G$-orbits of $P$-invariant nontrivial characters (we are reasoning as in the third paragraph of (2.b) in the proof of Theorem \ref{thmA:case6}). Since each character in one these $G$-orbits has at least 3 characters in the principal block above it by Theorem \ref{thm:thetatheorem}, we obtain a contradiction. 

\smallskip

(2.d)  Finally we consider the case where $k_0(B_0(G/N))=6$. Reasoning as in case (2.c) we can assume that $P$ acts transitively on the set of simple direct factors of $N$.
 In this case $[PN:P'N]=9$ by Theorem \ref{thmA:case6}. Suppose that Theorem \ref{thm:conditionsonsimples}(a1) holds for $S$ and let $\theta_1,\theta_2$ as in Theorem \ref{thm:conditionsonsimples}(a1). Let $\tilde{\theta_i}\in{\rm Irr}_{p'}(B_0(N))$ be the product of the copies of $\theta_i$, so $\tilde{\theta}_1$ and $\tilde{\theta}_2$ are $P$-invariant (so they extend to $PN$ by \cite[Cor.~8.16]{Isaacs76}) and not $G$-conjugate. By Lemma \ref{lem:Murai} and Theorem \ref{thm:thetatheorem} each of these characters has at least 3 irreducible characters in ${\rm Irr}_{p'}(B_0(G))$ lying above them, a contradiction.
Finally suppose that Theorem \ref{thm:conditionsonsimples}(a2) holds for $S$ and let $\theta\in{\rm Irr}_{p'}(B_0(S))$ be the character given by the statement of Theorem \ref{thm:conditionsonsimples}(a2).  
Let $\tilde{\theta}\in{\rm Irr}_{p'}(B_0(N))$ be the product of copies of $\theta$, so that $\tilde{\theta}$ is $P$-invariant. Notice then that $G_{\tilde \theta}$ acts transitively on the simple direct factors of $N$. By Theorem \ref{tensorinduction} we also know that $\tilde \theta$ extends to some $\phi \in {\rm Irr}(B_0(G_{\tilde \theta}))$. By Lemma \ref{lem:extension} we have that $k_0(B_0(G_{\tilde \theta})|\tilde{\theta})\geq k_0(B_0(G_{\tilde \theta}/N))$. Since $|G_{\tilde{\theta}}:N|_3>3$, then  $k_0(B_0(G/N))>3$ by \cite[Thm.~C]{NST18}, and this gives a contradiction as $B_0$ contains at most 3 characters of degree coprime to $p$ not lying over $1_N$. 
   \end{proof}
   
\subsection*{Acknowledgments} The authors are indebted to Miquel Mart\'inez for a careful reading of a previous version of this preprint and for detecting a missing case in the proof of Theorem \ref{thmA:case9} and to Gunter Malle for providing many useful comments that have greatly improved the readability of this article. 
We would also like to thank Gustavo Fern\'andez-Alcober, Alex Moret\'o and Benjamin Sambale for useful discussions on $p$-groups.
 In particular, we are grateful to Benjamin Sambale (and incidentally to Heiko Dietrich and Bettina Eick) for providing a reference for the asymptotic behavior of the number of isomorphism classes of (metabelian) $3$-groups with abelianization of size $9$.

\end{document}